\documentclass[a4paper,11pt]{article}
\usepackage[french,english]{babel}
\usepackage[babel=true,kerning=true]{microtype}
\usepackage[utf8]{inputenc} 
\usepackage{graphicx}
\usepackage{amssymb}
\usepackage{amsmath}
\usepackage{amsthm}
\usepackage{geometry}
\usepackage{tikz}
\usepackage{pgfkeys}
\usepackage{stmaryrd}
\usepackage{ifthen}
\usepackage{caption}
\usepackage{hyperref}
\usepackage{enumerate}

\newcommand{\abs}[1]{\left\lvert #1 \right\rvert}
\newcommand{\Int}{\mathrm{Int}~}
\newcommand{\norm}[1]{\left\lVert #1 \right\rVert}
\newcommand{\setof}[2]{\left\{ #1 ~ \middle\arrowvert ~ #2 \right\}}
\newcommand{\prodscal}[2]{\left\langle #1 ~ \middle\arrowvert ~ #2 \right\rangle}
\newcommand{\obstacle}[3]{
  \draw[color=black, fill=white] [domain={-rad(acos(#1-1))}:{rad(acos(#1-1))}, samples=100, smooth, variable=\x] ({-rad(acos(#1-1)) + #2}, {#3}) -- plot({\x + #2}, {rad(acos(#1-cos(\x r))) + #3}) --  plot({-\x + #2}, {-rad(acos(#1-cos(\x r))) + #3}) -- ({-rad(acos(#1-1)) + #2}, {#3});}

\hypersetup{
colorlinks=true,
pdfauthor={Mickaël Kourganoff},
pdftitle={Anosov geodesic flows, billiards and linkages},
pdfcreator=PDFLaTeX,
pdfproducer=PDFLaTeX,
linkcolor=[rgb]{0,0,0.7},
citecolor=[rgb]{0.7,0,0},
}

\usepackage{url}

\makeatletter
\def\url@leostyle{
 \@ifundefined{selectfont}{\def\UrlFont{\sf}}{\def\UrlFont{\small\ttfamily}}}
\makeatother
\newtheorem{prop}{Proposition}
\newtheorem{thm}[prop]{Theorem}
\newtheorem{lemme}[prop]{Lemma}
\newtheorem{fact}[prop]{Fact}

\theoremstyle{remark}

\newtheorem*{remarque}{Remark}

\newtheorem*{question}{Question}

\theoremstyle{definition}

\newtheorem{definition}[prop]{Definition}

\usetikzlibrary{shapes.geometric}
\usetikzlibrary{through,calc}
\usetikzlibrary{plotmarks}
\usetikzlibrary{arrows,decorations.markings}

\tikzset{
    scale plot marks/.is choice,
    scale plot marks/false/.code={
        \def\pgfuseplotmark##1{\pgftransformresetnontranslations\csname pgf@plot@mark@##1\endcsname}
    },
    scale plot marks/true/.style={},
    scale plot marks/.default=true
}

\tikzset{fleche/.style args={#1:#2}{ postaction = decorate,decoration={name=markings,mark=at position #1 with {\arrow[#2,scale=2]{>}}}}}

\tikzstyle{vertex}=[circle, fill=black, inner sep=0pt, minimum size=5pt]
\tikzstyle{fixed}=[rectangle, fill={rgb:blue,2;black,2}, inner sep=0pt, minimum size=5pt]

\newcommand{\manivellelabels}[8] 
{
\begin{tikzpicture}[auto=right, scale={#8/3}, transform shape]
    \node[fixed, label=left:$a_1$] (a1) at (3,0) {};
    \node[fixed, label=above:$a_2$] (a2) at (-3/2,{3*sqrt(3)/2}) {};
    \node[fixed, label=below:$a_3$] (a3) at (-3/2,{-3*sqrt(3)/2}) {};
    \node[vertex, label=below:$x$] (x) at ({#3*3}, {#4*3}) {};
    \pgfmathsetmacro{\longueurun}{#1*6}
    \pgfmathsetmacro{\longueurdeux}{#2*6}
    \node (c0) at (x) [circle, inner sep=0pt, minimum size=\longueurdeux cm] {};
    \node (c1) at (a1) [circle, inner sep=0pt, minimum size=\longueurun cm] {};
    \node (c2) at (a2) [circle, inner sep=0pt, minimum size=\longueurun cm] {};
    \node (c3) at (a3) [circle, inner sep=0pt, minimum size=\longueurun cm] {};
    \node (c4) at (0,0) [draw, circle, inner sep=0pt, minimum size=6 cm] {};
    \node [vertex, label=above:$p_1$] (p1) at (intersection #5 of c0 and c1) {};
    \node [vertex, label=above:$p_2$] (p2) at (intersection #6 of c0 and c2) {};
    \node [vertex, label=above:$p_3$] (p3) at (intersection #7 of c0 and c3) {};
    \draw[thick] (a1) to node [scale={0.8}] {$l_1$} (p1) to node [scale={0.8}] {$l_2$} (x);
    \draw[thick] (a2) to node [scale={0.8}] {$l_1$} (p2) to node [scale={0.8}] {$l_2$} (x);
    \draw[thick] (a3) to node [scale={0.8}] {$l_1$} (p3) to node [scale={0.8}] {$l_2$} (x);
\end{tikzpicture}
}

\newcommand{\anneaux}[3] 
{
\shorthandoff{:}
\begin{tikzpicture}[auto=right]
\begin{scope}[scale=#3]
  \begin{scope}
    \clip (-1.8,-2) rectangle (2.2, 2);
    \coordinate (a1) at (1,0);
    \coordinate (a2) at ({-1/2},{sqrt(3)/2});
    \coordinate (a3) at ({-1/2}, {-sqrt(3)/2});
    \path[fill=black!20] (a1) circle (#1);
    \path[fill=black!20] (a2) circle (#1);
    \path[fill=black!20] (a3) circle (#1);
  \begin{scope}
    \clip (a1) circle (#1);
    \clip (a2) circle (#1);
    \clip (a3) circle (#1);
    \path[fill=black!50, ultra thick, draw=black] (a1) circle (#1);
    \path[ultra thick, draw=black] (a2) circle (#1);
    \path[ultra thick, draw=black] (a3) circle (#1);
    \path[ultra thick, draw=black] (a1) circle (#2);
    \path[ultra thick, draw=black] (a2) circle (#2);
    \path[ultra thick, draw=black] (a3) circle (#2);
  \end{scope}
    \path[thick, fill=white] (a1) circle (#2);
    \path[thick, fill=white] (a2) circle (#2);
    \path[thick, fill=white] (a3) circle (#2);
    \path[draw=black] (a1) circle (#1);
    \path[draw=black] (a2) circle (#1);
    \path[draw=black] (a3) circle (#1);
    \path[draw=black] (a1) circle (#2);
    \path[draw=black] (a2) circle (#2);
    \path[draw=black] (a3) circle (#2);
    \foreach \point in {a1,a2,a3} {
	   \path [fill=black] (\point) circle (1.5pt);}
  \end{scope}
  \end{scope}
\end{tikzpicture}
\shorthandon{:}
}

\begin{document}

\selectlanguage{english}

\title{Anosov geodesic flows, billiards and linkages}
\author{Mickaël Kourganoff\footnote{UMPA, ENS Lyon \newline 46, allée d'Italie \newline 69364 Lyon Cedex 07 \newline France}}


\maketitle

\begin{abstract}
Any smooth surface in $\mathbb R^3$ may be flattened along the $z$-axis, and the flattened surface becomes close to a billiard table in $\mathbb R^2$. We show that, under some hypotheses, the geodesic flow of this surface converges locally uniformly to the billiard flow. Moreover, if the billiard is dispersive and has finite horizon, then the geodesic flow of the corresponding surface is Anosov. We apply this result to the theory of mechanical linkages and their dynamics: we provide a new example of a simple linkage whose physical behavior is Anosov. For the first time, the edge lengths of the mechanism are given explicitly.
\end{abstract}

\section{Introduction}

\subsection{Geodesic flows and billiards} \label{sectIntroFirstPart}

In 1927, Birkhoff~\cite{birknoff1927dynamical} noticed the following fact: if one of the principal axes of an ellipsoid tends to zero, then the geodesic flow of this ellipsoid tends, at least heuristically, to the billiard flow of the limiting ellipse. In 1963, Arnold~\cite{arnold1963small} stated that the billiard flow in a torus with strictly convex obstacles could be approximated by the geodesic flow of a flattened surface of negative curvature, which would consist of two copies of the billiard glued together, and suggested that this might imply that such a billiard would be chaotic. Later, Sinaï~\cite{sinai1970dynamical} proved the hyperbolicity of the billiard flow in this case, without using the approximation by geodesic flows. In the general case, the correspondance between billiards and geodesic flows of shrinked surfaces is well-known, but it is difficult to use in practice, and although the results which hold in both cases are similar, they require different proofs. One of the difficulties is the following: near tangential trajectories, some geodesics converge to ``fake'' billiard trajectories, which follow the boundary of the obstacle for some time and then leave (see Figure~\ref{figFake}).

\begin{figure}
\centering
\begin{tikzpicture}
    \path[fill=black!20] (-2,-2) rectangle (2,2);
    \path[draw=black, fill=white, fleche=0.9:black] (0,0) circle (1);
    \path[draw=black, fleche=0.5:black] (-2,-1) -- (0,-1);
    \path[draw=black, fleche=0.5:black] ({1/sqrt(2)},{1/sqrt(2)}) -- ({1/sqrt(2)-1},{1/sqrt(2)+1});
\end{tikzpicture}
\caption{A ``fake'' billiard trajectory, which is the limit of a sequence of geodesics in the flattening surface. The billiard (in grey) is in $\mathbb T^2$ and there is one circular obstacle (in white).} \label{figFake}
\end{figure}
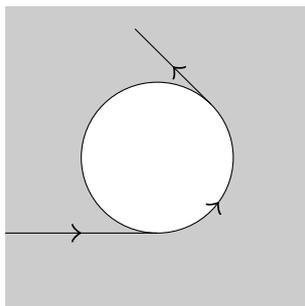

More precisely, for a given billiard $D \subseteq \mathbb T^2$ or $D \subseteq \mathbb R^2$, Birkhoff and Arnold's idea is to consider a surface $\Sigma$ in a space $E = \mathbb T^2 \times \mathbb R$ or $E = \mathbb R^3$, such that $D = \pi(\Sigma)$, where $\pi$ is the projection onto the two first coordinates; and then, to consider its image $\Sigma_\epsilon$ by a flattening map for $\epsilon > 0$:
\[ \begin{aligned} f_\epsilon : E & \to E \\ (x, y, z) & \mapsto (x, y, \epsilon z). \end{aligned} \]

The Euclidean metric of $\mathbb R^3$ induces a metric $h_\epsilon$ on $\Sigma_\epsilon$. It is convenient to consider the metric $g_\epsilon = (f_\epsilon)^*(h_\epsilon)$ on $\Sigma$, which tends to a degenerate $2$-form $g_0$ on $\Sigma$ as $\epsilon$ decreases to $0$. Thus, $(\Sigma, g_0)$ is not a Riemannian metric, but in many cases (for example, for the ellipsoid), it remains a metric space\footnote{In~\cite{burago1998uniform}, Burago, Ferleger and Kononeko used such degenerate spaces (\emph{Alexandrov spaces}) to estimate the number of collisions in some billiards.} and every billiard trajectory in $D$ corresponds to a geodesic in $(\Sigma, g_0)$. The Arzelà-Ascoli theorem guarantees that every sequence of unit speed geodesics $(\Gamma_n)$ in $(\Sigma, g_{\epsilon_n})$, with $\epsilon_n \to 0$, converges to a geodesic of $(\Sigma, g_0)$, up to a subsequence. In this paper, we prove a stronger version of this result.

Thus, from any given billard, Arnold constructs a surface which he flattens, so that its geodesic flow converges to the billiard flow. In this paper, we do the reverse: we prove that, under some natural hypotheses, the geodesic flow of \emph{any} given compact surface in $\mathbb R^3$, or $\mathbb T^2 \times \mathbb R$, flattening to a smooth billiard, converges locally uniformly to the billiard flow, away from grazing trajectories (Theorem~\ref{mainThm}). We also prove that, if the limiting billiard has finite horizon and is dispersive, then the geodesic flow in $(\Sigma, g_\epsilon)$ is Anosov for any small enough $\epsilon > 0$ (Theorem~\ref{thmAnosov}). In this case, it is well-known that the limiting billiard is chaotic, but the surface near the limit does not necessarily have negative curvature everywhere: some small positive curvature may remain in the area corresponding to the interior of the billiard, while the negative curvature concentrates in the area near the boundary. Since the limiting billiard has finite horizon, any geodesic falls eventually in the area of negative curvature, which guarantees that the flow is Anosov. The precise statements of our results are given in Section~\ref{sectMainResults}.

Other analogies have been made between billiards and smooth dynamical systems. In~\cite{turaev1998elliptic}, Turaev and Rom-Kedar showed that the billiard flow could be approximated in the $C^r$ topology by the behavior of a particle in $\mathbb R^2$ exposed to a potential field which explodes near the boundary\footnote{These systems are called \emph{soft billiards} in the literature: see also~\cite{balint2003correlation} for more details.}. In our situation, there is no potential and the particle has $3$ coordinates instead of $2$, but some of our techniques are similar to theirs.

Our setting has also much in common with the example of Donnay and Pugh~\cite{donnay2003anosov}, who exhibited in 2003 an embedded surface in $\mathbb R^3$ which has an Anosov flow. This surface consists of two big concentric spheres of very close radii, glued together by many tubes of negative curvature in a finite horizon pattern. In this surface, any geodesic eventually enters a tube and experiences negative curvature, while the positive curvature is small (because the spheres are big). However, in our situation, we may not choose the shape of the tubes and we need precise estimates on the curvature to show that the geodesic flow is Anosov.

\subsection{Mechanical linkages}
A mechanical linkage $\mathcal L$ is a physical system made of rigid rods joined by flexible joints. Mathematically, it is a graph $(V, E)$ with a length $l(e)$ associated to each edge $e$. Its \emph{configuration space} $\mathrm{Conf}(\mathcal L)$ is the set of all physical states of the system, namely:
\[ \mathrm{Conf}(\mathcal L) = \setof{\phi \in (\mathbb R^2)^V}{\forall (v, w) \in V, \lVert \phi(v) - \phi(w) \rVert = l(\phi(v), \phi(w))}. \]
It is an algebraic set in $(\mathbb R^2)^n = \mathbb R^{2n}$, where $n$ is the number of vertices. In the following, we will only consider linkages such that $\mathrm{Conf}(\mathcal L)$ is a smooth manifold in $(\mathbb R^2)^n$. It is the case for a generic choice of the edge lengths (see~\cite{jordan2001compact} for example).

In this paper, we are interested in the physical behavior of linkages when they are given an initial speed, without applying any external force. Of course, the dynamics depend on the distribution of the masses in the system: to simplify the problem, we will assume that the masses are all concentrated at the vertices of the graph. If one denotes the speed of each vertex by $v_i$, and the masses by $m_i$, the \emph{principle of least action} (see~\cite{arnol1978mathematical}) states that the trajectory between two times $t_0$ and $t_1$ will be a critical point of the kinetic energy
\[ \mathcal K = \frac{1}{2} \int_{t_0}^{t_1} \sum_{i=1}^n m_i v_i^2(t) dt, \]
which is also a characterization of the geodesics in the manifold $\mathrm{Conf}(\mathcal L)$ endowed with a suitable metric:
\begin{fact} \label{geodesicbehavior}
The physical behavior of the linkage $\mathcal L$, when it is isolated and given an initial speed, is the geodesic flow on $\mathrm{Conf}(\mathcal L) \subseteq (\mathbb R^2)^n = \mathbb R^{2n}$, endowed with the metric:
\[ g = \sum_{i=1}^n m_i (\mathrm dx_{2i-1}^2 + \mathrm dx_{2i}^2), \]
provided that the metric $g$ is nondegenerate.
In particular, if all the masses are equal to $1$, $g$ is the metric induced by the Euclidean $\mathbb R^{2n}$.
\end{fact}

\bigskip \noindent{\bf Anosov behavior.} We ask the following:
\begin{question}
Do there exist linkages with Anosov behavior?
\end{question}
The following theorem gives a theoretical answer to this question.

\begin{thm}
Let $(M, g)$ be any compact Riemannian manifold and $1 \leq k < +\infty$. Then there exists a linkage $\mathcal L$, a choice of masses, and a Riemannian metric $h$ on $M$, such that $h$ is $C^k$-close to $g$ and every connected component of $\mathrm{Conf}(\mathcal L)$ is isometric to $(M, h)$.
\end{thm}
\begin{proof}
Embed $(M, g)$ isometrically in some $\mathbb R^{2n}$: this is possible by a famous theorem of Nash~\cite{nash1956imbedding}. With another theorem of Nash and Tognoli (see~\cite{tognoli1973congettura}, and also~\cite{ivanov1982approximation}, page 6, Theorem 1), this surface is $C^k$-approximated by a smooth algebraic set $A$ in $\mathbb R^{2n}$, which is naturally equipped with the metric induced by $\mathbb R^{2n}$. The manifold $A$ is diffeomorphic to $M$, and even isometric to $(M, h)$ where the metric $h$ is $C^k$-close to $g$. In 2002, Kapovich and Millson~\cite{kapovich2002universality} showed that any compact algebraic set $B \subseteq \mathbb R^{2n}$ is exactly the partial configuration space of some linkage, that is, the set of the possible positions of a subset of the vertices; moreover, if $B$ a smooth submanifold of $\mathbb R^{2n}$, each connected component of the whole configuration space may be required to be smooth and diffeomorphic to $B$ (see also \cite{kourganoff2014universality} for more details). Thus, there is a linkage and a subset of the vertices $W$ such that the partial configuration space of $W$ is $A$: each component of the configuration space of this linkage, with masses $1$ for the vertices in $W$ and $0$ for the others, is isometric to the algebraic set $A$ endowed with the metric induced by $\mathbb R^{2n}$, which is itself isometric to $(M, h)$.
\end{proof}

In particular, there exist configuration spaces with negative sectional curvature, and thus with Anosov behavior. This answer is somewhat frustrating, as it is difficult to construct such a linkage with this method in practice, and it would have a high number of vertices anyway, at least several hundreds.

In the 1980's, Thurston and Weeks~\cite{thurston1984mathematics} pointed out that the configuration spaces of quite simple linkages could have an interesting topology, by introducing the famous \emph{triple linkage} (see Figure~\ref{triplelinkage}): they showed that, for some choice of the lengths, its configuration space could be a compact orientable surface of genus $3$. Later, Hunt and McKay~\cite{hunt2003anosov} showed that there exists a set of lengths for the triple linkage such that the configuration space is close to a surface with negative curvature everywhere (except at a finite number of points), and thus its geodesic flow is Anosov.

\begin{figure}[!ht]
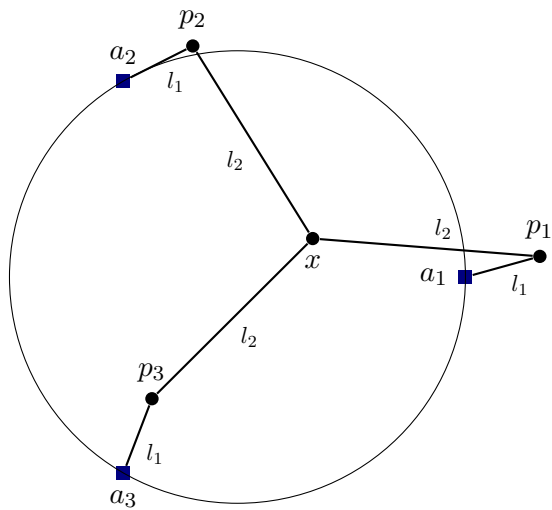

\centering
\manivellelabels{1/3}{1}{0.33}{0.17}{2}{1}{1}{3}
\caption{Thurston and Week's triple linkage. The vertices $a_1, a_2, a_3$ are pinned down to the plane, while the vertices $x, p_1, p_2, p_3$ may move.} \label{triplelinkage}
\end{figure}

\begin{figure}[!ht]
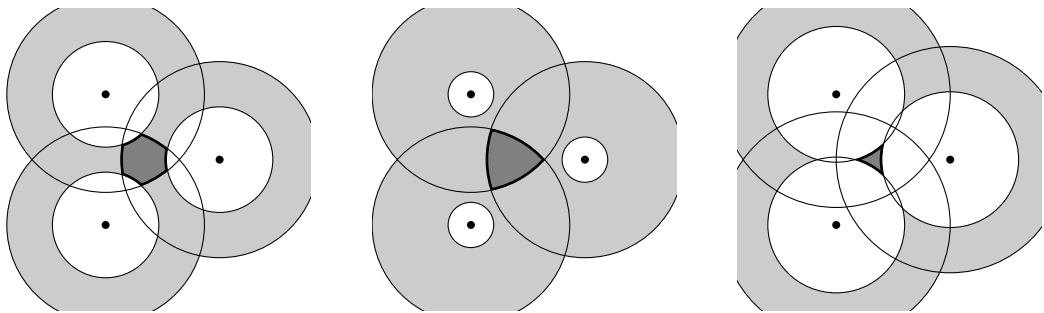

\centering
\begin{tabular}{ccc}
 \anneaux{1.3}{0.7}{1}
 & \anneaux{1.3}{0.3}{1}
 & \anneaux{1.5}{0.9}{1}
\end{tabular}
\caption{The workspace of the central vertex in Thurston's triple linkage is the intersection of three annuli centered at $a_1$, $a_2$ and $a_3$ (here, in dark grey). When the lengths $l_1$ and $l_2$ vary, the intersection may take different shapes. In the case on the left, the configuration space has genus $3$: it is what McKay and Hunt were interested in. On the right, the configuration space is a sphere.} \label{figWorkspace}
\end{figure}

\bigskip \noindent{\bf Asymptotic configuration spaces.} The computation of the curvature of a given configuration space is impossible in practice, most of the time. Thus, the idea of Hunt and McKay was to make some of the lengths tend to $0$, while the masses are fixed ($0$ for the vertex at the center and $1$ for the others), and to consider the limit of $\mathrm{Conf}(\mathcal L_\epsilon)$. At the limit, the surface is not the configuration space of a physical system anymore (it is called an \emph{asymptotic configuration space}), but it is easier to study because the equations are simpler. In the case of the triple linkage, the miracle is that the limit surface is Schwarz's well-known ``P surface'' in $\mathbb T^3$, defined by $\sum_{i=1}^{3} \cos x_i = 0$, which has negative curvature except at a finite number of points, and thus an Anosov geodesic flow. The structural stability of Anosov flows allows the authors to conclude that the configuration space of $\mathcal L_\epsilon$ for a small enough $\epsilon$ has an Anosov geodesic flow. In particular, one does not know how small $\epsilon$ has to be for $\mathcal L_\epsilon$ to be an Anosov linkage.

This technique may be applied to other linkages. For example, in 2013, Policott and Magalh{\~a}es~\cite{magalhaes2013geometry} tried to see what happened with the ``double linkage'', an equivalent of the triple linkage but with only two articulated arms (also called ``pentagon''). But the asymptotic configuration space in that case has both positive and negative curvature and it is impossible to conclude that the geodesic flow is Anosov, although their computer simulation suggests that it should be the case. In fact, since Hunt and McKay's example, no other linkage has been proved mathematically to be Anosov.

\bigskip \noindent{\bf Linkages and billiards.} In this paper, we provide a new example of an Anosov linkage, as an application of our results. To understand the link between linkages and billiards, consider Thurston's triple linkage, where all vertices have mass $0$ except the central vertex which has mass $1$. The workspace of the central vertex is a hexagon, and its trajectories are obviously straight lines in the interior of the workspace, but what happens physically when the vertex hits the boundary of the workspace? It turns out that it reflects by a billiard law. In fact, when the masses of the non-central vertices are a small $\epsilon^2 > 0$, the configuration space is equipped with the metric of a flattened surface $\Sigma_\epsilon$.

However, it may happen that the workspace of some vertex is a dispersive billiard, while the geodesic flow in the configuration space (with a small parameter $\epsilon$) is not Anosov. For example, consider Thurston's triple linkage in the case on the right of Figure~\ref{figWorkspace}. Then the workspace of the central vertex $x$ is a non-smooth dispersive billiard -- a triangle with negatively curved walls -- but the configuration space is topologically a sphere, so its geodesic flow cannot be Anosov. In fact, the corners of the billiard concentrate the positive curvature of the configuration space when it flattens.

\begin{figure}[!ht]
\centering
\includegraphics[height=200pt]{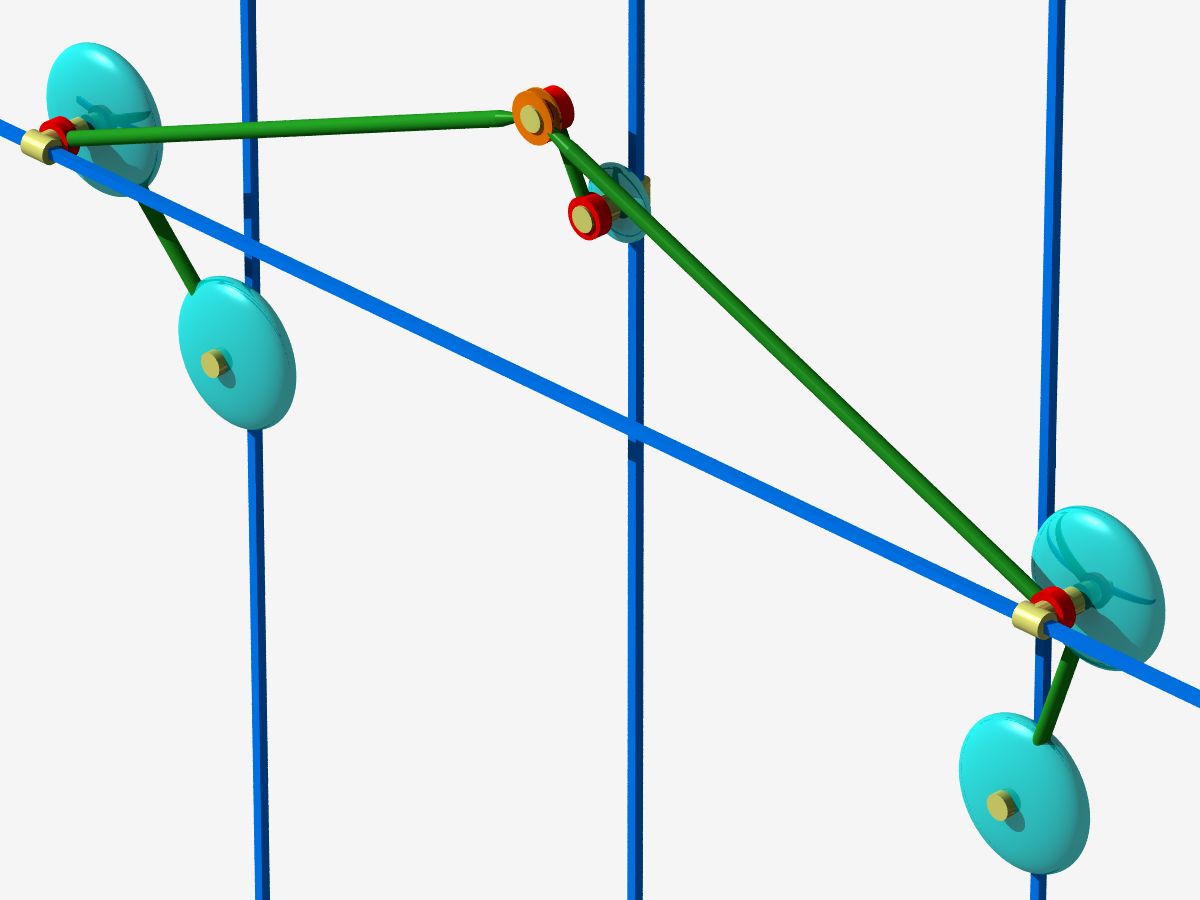}
\caption{A physical realization of our Anosov linkage.} \label{figPhysical}
\end{figure}

In our example, the billiard is not the workspace of a single vertex: it is the partial configuration space of four vertices, that is, the set of the possible positions of these vertices. It is \emph{a priori} a subset of $(\mathbb R^2)^4$, but in this particular case, it turns out that it may be seen as a subset of $\mathbb T^2$. The configuration space $\mathrm{Conf}(\mathcal L)$, in turn, may be seen as an immersed surface in $\mathbb T^2 \times \mathbb R$ which flattens to the billiard table as one of the masses tends to $0$.

\begin{figure}[!ht]
\centering
\includegraphics[height=200pt]{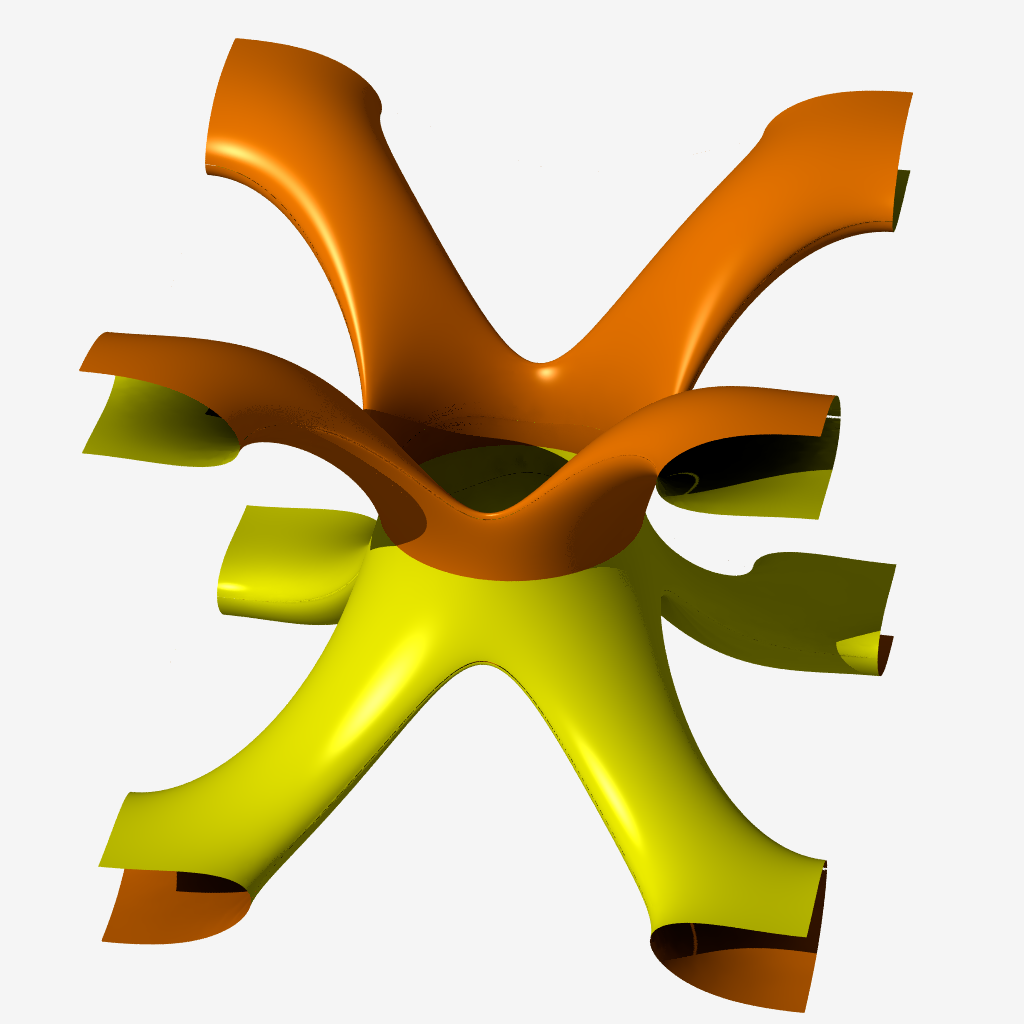}
~ \includegraphics[height=200pt]{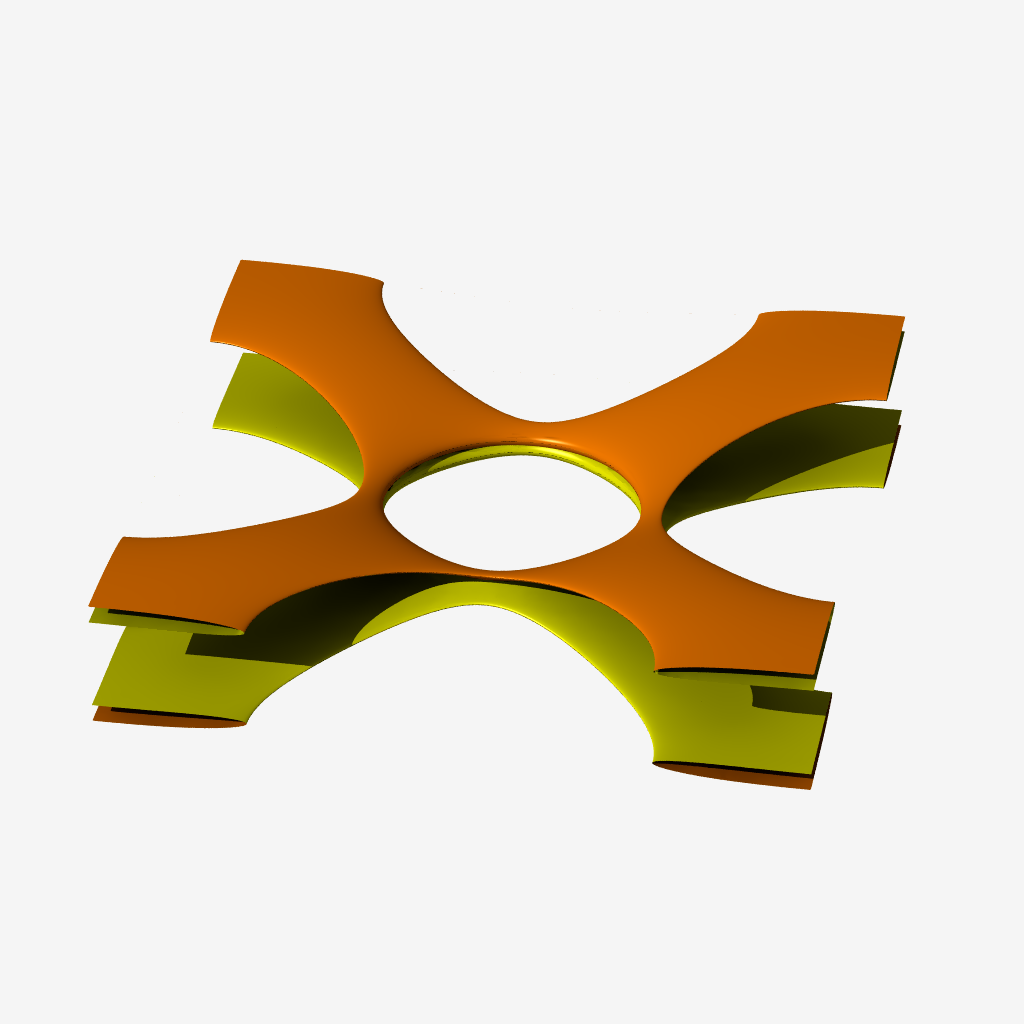}
\caption{On the left, the configuration space of our linkage in $\mathbb T^2 \times \mathbb R$, where $\mathbb T^2$ is horizontal and $\mathbb R$ vertical. It is a surface of genus 7, with a self-intersection (at the center of the picture). On the right, the flattened configuration space, which is close to a billiard table with finite horizon (see also Figure~\ref{figBilliard}).} \label{figConfSpace}
\end{figure}

Notice that our example is \emph{not} an asymptotic linkage, in the following sense: there is a whole explicit range of values for the edge lengths such that the linkage has an Anosov behavior. This is the first time that a linkage with explicit lengths is proved to be Anosov.

However, one mass has to be close to $0$ and our theorem does not say explicitly how close it has to be. Maybe this linkage is, in fact, Anosov even when the mass is equal to~$1$.

\bigskip \noindent{\bf A realistic physical system.} Similarly to Hunt and McKay~\cite{hunt2003anosov}, we insist on the fact that our linkage is realistic from a physical point of view. For example, it is possible to add small masses to the rods and to the central vertex without losing the Anosov property (using the structural stability of Anosov flows). See Hunt and McKay's article for more details about this aspect.

\section{Main results} \label{sectMainResults}

In this paper, we consider only smooth billiards in $B = \mathbb R^2$ or $\mathbb T^2$: we do not allow corners, to avoid the problems discussed in the introduction.

\begin{definition}
A smooth billiard table $D \subseteq B$ is the closure of an open set in $B$ such that $\partial D$ is a smooth manifold of dimension $1$ without boundary: in other words, each component $\partial_iD$ of $\partial D$ is the image of a smooth embedding $\Gamma_i : \mathbb T^1 \to B$. The curves $\Gamma_i$ are called the \emph{walls} of $D$. For each $\Gamma_i$, we define $T_i$ the unit tangent vector and $N_i$ the unit normal vector to the curve $\Gamma_i$ pointing towards $\Int D$. The curvature of $\Gamma_i$ is $\prodscal{\frac{dT_i}{ds}}{N_i}$. For example, the walls of a disc are positively curved, while the walls of its complementary set are negatively curved.
\end{definition}

Consider a compact surface $\Sigma$ immersed in $E = B \times \mathbb R$, whose canonical basis is written $(e_x, e_y, e_z)$. With the notations of Section~\ref{sectIntroFirstPart}, we denote by $\Phi^\epsilon_t$ the geodesic flow on $(\Sigma, g_\epsilon)$ and by $\Psi_t$ the billiard flow on $D$.

Recall that $\pi : E \to B$ is the projection onto the two first coordinates, while $f_\epsilon$ is a contraction along the $z$-axis. They induce mappings on the unit tangent bundles:
\[
\left\{
\begin{aligned}
\pi_* : T^1\left(\Sigma_\epsilon \cap \pi^{-1}\left(\Int D\right)\right) & \to T^1 D
\\ (q, p) & \mapsto \left(\pi(q), \frac{D_q \pi(p)}{\norm{D_q \pi(p)}}\right)
\end{aligned} \right.
\]
and
\[ \left\{
\begin{aligned}
(f_\epsilon)_* : T^1\Sigma & \to T^1\Sigma_\epsilon
\\ (q, p) & \mapsto \left(f_\epsilon(q), \frac{D_q f_\epsilon(p)}{\norm{D_q f_\epsilon(p)}}\right).
\end{aligned}
\right.
\]

Consider also the set $A$ of all $(t, q, p) \in \mathbb R \times T^1 \Sigma$ such that $\pi_*(q, p)$ and $\Psi_t \circ \pi_*(q, p)$ belong to $\mathrm{Int}~D$, and that the billiard trajectory between $\pi_*(q, p)$ and $\Psi_t \circ \pi_*(q, p)$ does not have a tangential collision with a wall of the billiard. Notice that $A$ is an open dense subset of $\mathbb R \times T^1 \Sigma$.

\begin{thm} \label{mainThm}
Assume that
\begin{enumerate}
\item \label{pasverticalsaufaubord} for all $q \in \pi^{-1}(\mathrm{Int}~D) \cap \Sigma$, $e_z \not\in T_q\Sigma$;
\item \label{courbureverticalenonnulle} for all $q \in \pi^{-1}(\partial D) \cap \Sigma$, the curvature of $\Sigma \cap V$ is nonzero at $q$, where $V$ is a neighborhood of $q$ in the affine plane $q + \mathrm{Vect} (e_z, (T_q\Sigma)^\bot)$.
\end{enumerate}

Then:
\[
\begin{aligned}
A & \to T^1(\mathrm{Int}~D) \\
(t, q, p) & \mapsto \pi_* \circ \Phi_t^\epsilon (q, p)
\end{aligned} \] converges uniformly on every compact subset of $A$ to 
\[
\begin{aligned}
A & \to T^1(\mathrm{Int}~D) \\
(t, q, p) & \mapsto \Psi_t \circ \pi_* (q, p)
\end{aligned}
\] as $\epsilon \to 0$.
\end{thm}

\begin{remarque}
If $\Sigma$ is a connected compact surface embedded in $\mathbb R^3$, with positive curvature everywhere, then the two assumptions of Theorem~\ref{mainThm} are automatically satisfied, and the description of $A$ is simpler.
\end{remarque}

On the other side, concerning dispersing billiards, we prove:

\begin{thm} \label{thmAnosov}
In addition to the two hypotheses of Theorem~\ref{mainThm}, assume that $B = \mathbb T^2$ and:
\begin{enumerate}
\setcounter{enumi}{2}
\item \label{billardcourburenegative} the walls of the billiard $D$ have negative curvature;
\item \label{horizonfini} the billiard $D$ has finite horizon: it contains no geodesic of $\mathbb T^2$ with infinite lifetime in the past and the future.
\end{enumerate}

Then for any small enough $\epsilon > 0$, the geodesic flow on $(\Sigma_\epsilon, h_\epsilon)$ is Anosov.
\end{thm}

In the proofs of Theorems~\ref{mainThm} and~\ref{thmAnosov}, we will assume that $\Sigma$ is \emph{embedded} in $E$, to simplify the notations, but the same proof works for the immersed case. In the case that $B = \mathbb T^2$, we will see $D$ as a periodic billiard in the universal cover $\mathbb R^2$, and $\Sigma$ as a periodic surface in $\mathbb R^3$.

In the last section, we apply Theorem~\ref{thmAnosov} to give an example of an Anosov linkage (see Figure~\ref{figureExemple}). All the vertices except one have only one degree of freedom and move on a straight line. This may be realized physically by \emph{prismatic joints} -- or, if one wants to stick to the traditional definition of linkages, it is possible to use Peaucellier's straight line linkage, or to approximate the straight lines by portions of arcs of large radius.

\begin{figure}[!ht]
\centering
\begin{tikzpicture}[scale=2]
\node[inner sep=0pt, label=below:{$x=-2$}] (v1) at (-2, -2) {};
\node[inner sep=0pt] (v2) at (-2, 2) {};
\node[inner sep=0pt, label=below:{$x=2$}] (v3) at (2, -2) {};
\node[inner sep=0pt] (v4) at (2, 2) {};
\node[inner sep=0pt, label=below:{$x=0$}] (v5) at (0, -2) {};
\node[inner sep=0pt] (v6) at (0, 2) {};
\node[inner sep=0pt, label=right:{$y=0$}] (v7) at (3, 0) {};
\node[inner sep=0pt] (v8) at (-3, 0) {};
\node[vertex, label=above:{$(a,0)$}] (v9) at ({-2-sqrt(3)/2}, 0) {};
\node[vertex, label=right:{$(0,f)$}] (v10) at (-2, -0.5) {};
\node[vertex, label=above:{$(b,0)$}] (v11) at (2.5, 0) {};
\node[vertex, label=left:{$(0,g)$}] (v12) at (2, -{sqrt(3)/2}) {};
\node[vertex, label=left:{$(d, e)$}] (v13) at ({(-sqrt(3)/2+0.5)/2},1) {};
\node[vertex, label=right:{$(0,c)$}] (v14) at (0,1.5) {};
\draw (v1) -- (v2);
\draw (v3) -- (v4);
\draw (v5) -- (v6);
\draw (v7) -- (v8);
\draw (v10) to node [below] {$1$} (v9) to node [below] {$l$} (v13) to node [below] {$l$} (v11) to node [right] {$1$} (v12);
\draw (v13) to node [left] {$r$} (v14);
\end{tikzpicture}
\caption{Mathematical description of our Anosov linkage.} \label{figureExemple}
\end{figure}
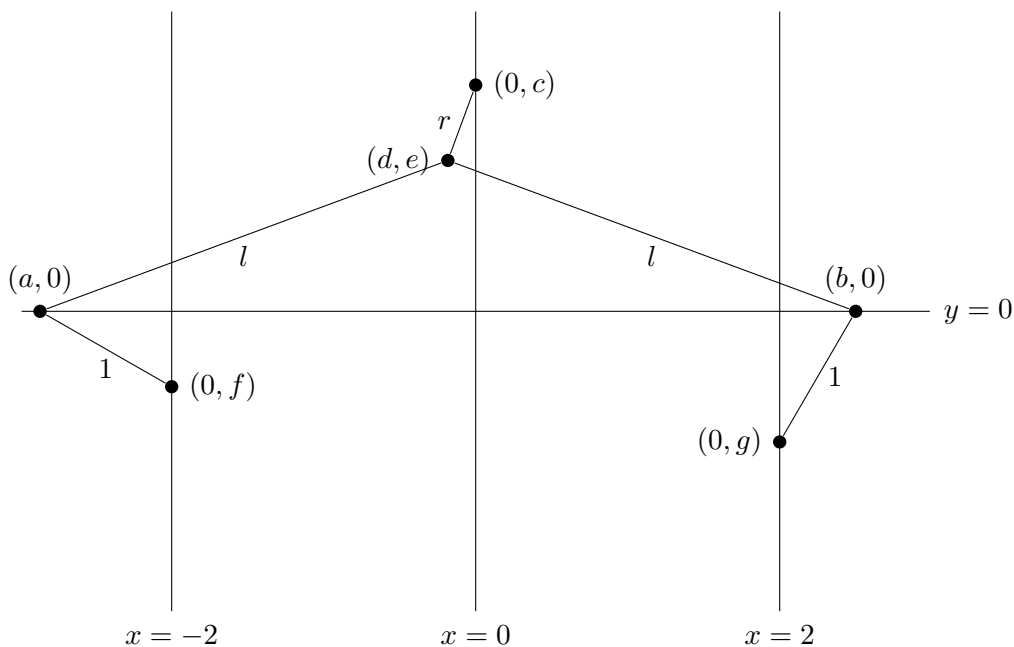

\begin{thm} \label{thmLinkage}
In the linkage of Figure~\ref{figureExemple}, choose the lengths of the rods such that $l + r > 3$, $l < 3$, $(l-2)^2 + r^2 < 1$ and $r < 1/2$. The mass at $(a, 0), (0, f), (b, 0)$ and $(0, g)$ is $1$, the mass at $(0,c)$ is $\epsilon^2$, while the mass at $(d, e)$ is $0$. Then for any sufficiently small $\epsilon > 0$, the geodesic flow on the configuration space of the linkage is Anosov.
\end{thm}

\bigskip \noindent{\bf Structure of the proofs.} The main tool to study the geodesic flow is the geodesic equation which involves the position $q$, the speed $p$, and the normal vector $N$ to $\Sigma$:

\begin{equation} \label{eqDerivees} \dot{p} = - N(q) \prodscal{DN(q) \cdot p}{p}. \end{equation}
It is simply obtained by taking the derivative of the equation
\[ \prodscal{p}{N} = 0. \]
Equation~\ref{eqDerivees} involves the second fundamental form, which is closely linked to the curvature of $\Sigma$: in Section~\ref{sectThmConvergence}, we make precise estimates on the second fundamental form, study nongrazing collisions with the walls of the billiard (Lemma~\ref{nonGrazing}) and prove the uniform convergence of the flow (Theorem~\ref{mainThm}). In Section~\ref{sectThmAnosov}, we prove that the geodesic flow is Anosov (Theorem~\ref{thmAnosov}): for this, we also need to study grazing trajectories (Lemma~\ref{grazing}), and examine the solutions of the Ricatti equation
\[ \left\{
\begin{aligned}
u (0) & = 0
\\ u'(t) & = - K(t) - u^2(t)
\end{aligned} \right. \]
where $K$ is the Gaussian curvature of $\Sigma$. In the last section, we give a proof of Theorem~\ref{thmLinkage}, which mainly consists in checking that the configuration space of the linkage of Figure~\ref{figureExemple} satisfies the assumptions of Theorem~\ref{thmAnosov}.

\begin{figure}[!ht]
\centering
\includegraphics[height=200pt]{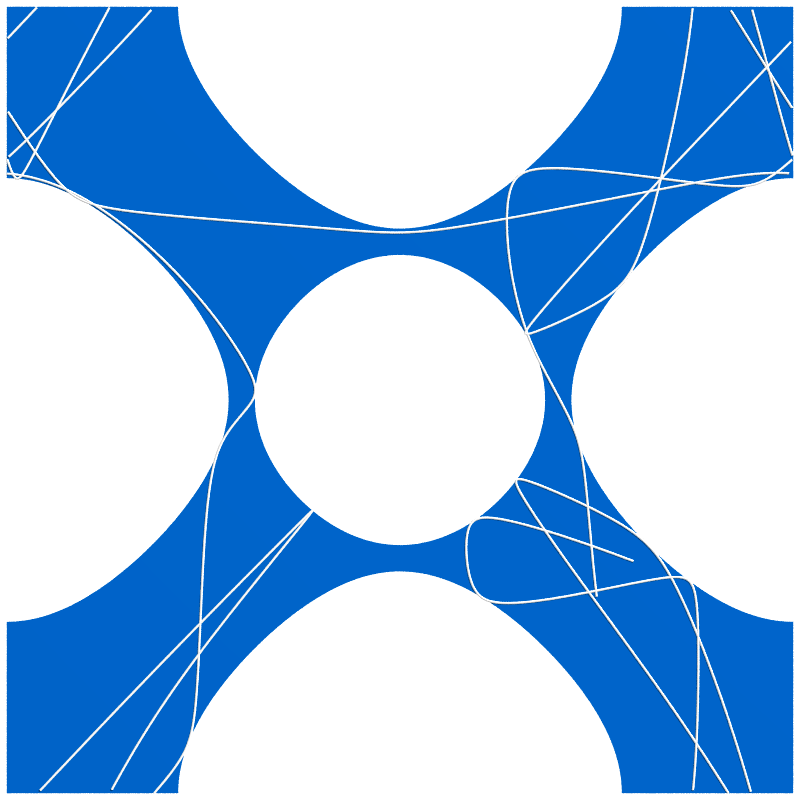}
\caption{The projection of the geodesics of the configuration space onto the billiard table. Theorem~\ref{mainThm} states that the nongrazing geodesics are close to billiard trajectories. For Theorem~\ref{thmAnosov}, we also need to study the grazing trajectories, whose behavior is described by Figure~\ref{figFake}.} \label{figBilliardGeodesics}
\end{figure}

\bigskip \noindent{\bf Some questions.} For technical reasons, we had to introduce Assumption~\ref{courbureverticalenonnulle} in the statements of Theorems~\ref{mainThm} and~\ref{thmAnosov}, but is it necessary? One may also wonder whether these theorems generalize to surfaces immersed in higher-dimensional spaces, or if the convergence in Theorem~\ref{mainThm} holds for the $C^k$ topology.

\section{Proof of Theorem~\ref{mainThm}} \label{sectThmConvergence}

In $\Sigma \cap \pi^{-1}(\mathrm{Int}~D)$, $g_\epsilon$ converges smoothly to a flat metric, so the geodesic flow converges smoothly to the billiard flow. Hence, the difficulty of the proof concentrates at the boundary of the billiard table: there, we have to show that the geodesic flow satisfies a billiard reflection law at the limit (Proposition~\ref{nonGrazing}). For this, we will need some estimates on the second fundamental form of the surface near the boundary.

First, let us fix some notations.

\begin{definition} \label{defNormal}
Given $\epsilon > 0$, we may choose a normal vector $N^\epsilon$ on any simply connected subset of $\Sigma_\epsilon$. We will always assume implicitly that such a choice of orientation has been made: since we work locally on the surface, it is not necessary to have a global orientation.

Consider $N_x^\epsilon, N_y^\epsilon, N_z^\epsilon$ the three components of $N^\epsilon$ in $\mathbb R^3$. Thus for $q \in \Sigma$:

\[ N_x^\epsilon (f_\epsilon(q)) = \frac{N_x^1}{\sqrt{(N_x^1)^2 + (N_y^1)^2 + \frac{1}{\epsilon^2} (N_z^1)^2}}, \quad N_y^\epsilon (f_\epsilon(q)) = \frac{N_y^1}{\sqrt{(N_x^1)^2 + (N_y^1)^2 + \frac{1}{\epsilon^2} (N_z^1)^2}} \] and \[ N_z^\epsilon (f_\epsilon(q)) = \frac{\frac{1}{\epsilon} N_z^1}{\sqrt{(N_x^1)^2 + (N_y^1)^2 + \frac{1}{\epsilon^2} (N_z^1)^2}}. \]

We shall often simply write $N$ instead of $N^\epsilon$, when there is no possible confusion.

Finally, define \[H (f_\epsilon(q)) = N_z^1(q). \] The quantity $H (f_\epsilon(q))$ has the advantage of being independent of $\epsilon$, contrary to $N^\epsilon (f_\epsilon(q))$.

For all $q \in \Sigma_\epsilon$, we know that $\pi(q) \in \partial D$ if and only if $N_z(q) = 0$, or equivalently, $H(q) = 0$. This gives us two notions of ``being close to the boundary'': for all $\epsilon, \delta, \nu \in (0, 1)$, we define
\[ V_\nu^\epsilon = \setof{q \in \Sigma_\epsilon}{ \abs{N_z^\epsilon(f_\epsilon(q))} < 1 - \nu } \]
and
\[ Z_\delta^\epsilon := \setof{ q \in \Sigma_\epsilon }{ \abs{H(q)} \leq \delta }. \]

To simplify the notations, we will often omit the $\epsilon$ and simply write $V_\nu$ and $Z_\delta$. Notice that for any $\delta$ and $\nu$, when $\epsilon$ is sufficiently small, we have $V_\nu \subseteq Z_\delta$, because the metric tends to a flat one outside $Z_\delta$.
\end{definition}

\begin{definition}[Darboux frame] \label{darboux}
For any unit speed curve $\Gamma : [0,1] \to \Sigma$, we define $T = \Gamma'(s)$ the tangent vector. The normal vector $N$ is the unit normal to $(T_{\Gamma(s)} \Sigma)$. Finally, the \emph{normal geodesic vector} $G$ is defined by $G = N \wedge T$.

In this frame, there exist three quantities $\gamma_N^\epsilon$ (normal curvature), $\gamma_G^\epsilon$ (geodesic curvature) and $\tau_G^\epsilon$ (geodesic torsion), also written simply $\gamma_N$, $\gamma_G$ and $\tau_G$, such that
\[
\begin{aligned}
\frac{\mathrm dT}{\mathrm ds} & = \gamma_G G + \gamma_N N
\\ \frac{\mathrm dG}{\mathrm ds} & = \gamma_G T + \tau_G N
\\ \frac{\mathrm dN}{\mathrm ds} & = \gamma_N T + \tau_G G
\end{aligned}
\]

The (traditional) curvature of $\Gamma$ considered as a curve in $\mathbb R^3$ is $k = \norm{\frac{\mathrm dT}{\mathrm ds}}$. Thus, if $k \neq 0$, writing $n = \frac{\mathrm dT / \mathrm ds}{\norm{\mathrm dT / \mathrm ds}}$, we obtain:
\begin{equation} \label{eqk} \gamma_N = k \prodscal{N}{n}, \end{equation}
and in particular:
\begin{equation} \label{eqk2} \abs{\gamma_N} = k \sqrt{1 - \prodscal{N}{T \wedge n}^2}. \end{equation}
For example, if $\Gamma$ is the intersection of $\Sigma$ with a plane $P$, $T \wedge n$ has the same direction as the normal vector of $P$, so it is convenient to use Equation~\ref{eqk2}.

Notice that the normal curvature at $s$ only depends on $\Gamma(s)$ and $\Gamma'(s)$: thus we may write $\gamma_N(q, p)$ for $(q, p) \in T^1\Sigma$. Moreover, we have the relation:

\[ \gamma_N(q, p) = \prodscal{ DN(q) \cdot p } {p}. \]
\end{definition}

For any $q \in \Sigma_\epsilon$, $\gamma_+^\epsilon(q)$ and $\gamma_-^\epsilon(q)$ (sometimes written simply $\gamma_+$ and $\gamma_-$) are the principal curvatures of $\Sigma_\epsilon$ at $q$. They correspond respectively to the maximum and minimum normal curvatures at $f_\epsilon(q)$.

$K^\epsilon(q) = \gamma_+^\epsilon(q) \gamma_-^\epsilon(q)$ is the Gaussian curvature of $\Sigma_\epsilon$ at $q$.

We can now make a first remark:
\begin{fact} \label{submersion}
For any small enough $\delta > 0$, $H|_{Z_\delta}$ is a submersion from $Z_\delta$ to $\mathbb R$.
\end{fact}
\begin{proof}
Let $q \in \pi^{-1}(\partial D) \cap \Sigma$ and consider a curve $\gamma$ which parametrizes the section of $\Sigma$ by the plane containing the directions $(Oy)$ and $(Oz)$, with $\gamma(0) = q$. Assumption~\ref{courbureverticalenonnulle} of the theorem implies that $\gamma$ has nonzero curvature at $q$, and thus $\prodscal{D_{\gamma(t)}H(\gamma'(t))}{\gamma'(t)}$ is nonzero for any small enough $t$. Therefore, $H$ is a submersion from $Z_\delta$ to $\mathbb R$.
\end{proof}

\begin{lemme} \label{lemmeCourbureRestePositive}
Let $(q, p) \in T^1\Sigma$ and $(q^\epsilon, p^\epsilon) = (f_\epsilon)_*(q, p)$. If $\gamma_N^\epsilon(q^\epsilon, p^\epsilon) \neq 0$ for some $\epsilon > 0$, then the sign of $\gamma_N^\epsilon(q^\epsilon, p^\epsilon)$ is the same for all $\epsilon > 0$.
\end{lemme}
\begin{proof}
Let $\Gamma: [-1,1] \to \Sigma$ be any curve such that $(\Gamma(0), \Gamma'(0)) = (q, p)$, and consider $\Gamma^\epsilon = f_\epsilon \circ \Gamma$ for $\epsilon > 0$. Writing $T^\epsilon$ its tangent vector, the assumption implies that $\prodscal{\frac{d T^\epsilon}{dt}}{N^\epsilon}$ is nonzero at $t=0$ for some $\epsilon$, which means that $\frac{dT^\epsilon}{dt} \not\in T_{\Gamma^\epsilon(0)} \Sigma_\epsilon$. Obviously, this property does not depend on $\epsilon$, so $\gamma_N^\epsilon(q^\epsilon, p^\epsilon)$ is nonzero for all $\epsilon$. By continuity, $\gamma_N^\epsilon(q^\epsilon, p^\epsilon)$ does not change sign.
\end{proof}

\begin{lemme} \label{lemmeCourburePositive}
Let $\alpha \in (0, 1)$, and $q_0 \in \Sigma$ such that $\pi(q_0) \in \partial D$. We assume that $N(q_0)$ is directed towards the exterior of the billiard table $D$, and (up to a rotation of axis $e_z$) that $N(q_0) = -e_y$. Then there exists $r > 0$ and $\epsilon_0$ such that for all $\epsilon \leq \epsilon_0$, all $q \in B(f_\epsilon(q_0), r)$, and all $p \in T_{q}^1 \Sigma_\epsilon$: $\gamma_N(q, p) > 0$ whenever $\abs{p_x} \leq \alpha$.
\end{lemme}
\begin{proof}
By Assumption~\ref{courbureverticalenonnulle} of Theorem~\ref{mainThm}, we know that $\gamma_N(q_0, e_z) > 0$. We define $W_\eta = \setof{ p \in T_{q_0}^1\Sigma}{\abs{p_x} \leq \eta}$. By continuity of $\gamma_N$, there exists $\eta > 0$ such that, for all $p \in W_\eta$, $\gamma_N(q_0, p) > 0$.

Write $W_\eta^\epsilon = (f_\epsilon)_* (W_\eta)$. Notice that for all $p \in T^1_{q_0} \Sigma$, writing $(f_\epsilon)_*(q_0, p) = (f_\epsilon(q_0), p^\epsilon)$:
\[ p^\epsilon_x = \frac{p_x}{\sqrt{p_x^2 + \epsilon^2 p_z^2}} \quad \text{and} \quad p^\epsilon_z = \frac{\epsilon p_z}{\sqrt{p_x^2 + \epsilon^2 p_z^2}}. \]
Thus for a small enough $\epsilon$, $W_\eta^\epsilon$ contains all $p \in T_{q_0}^1\Sigma_\epsilon$ for which $\abs{p_x} \leq \alpha$. We denote such an $\epsilon$ by $\epsilon_0$. By Lemma~\ref{lemmeCourbureRestePositive}, $\gamma_N > 0$ on $W_\eta^{\epsilon_0}$. Again by continuity, the property $\gamma_N > 0$ extends to a small neighborhood of the form $\setof{(q, p) \in T^1\Sigma_{\epsilon_0}}{ q \in B(q_0, r'), ~ \abs{p_x} \leq \alpha}$ for some $r' > 0$.

Finally, we use Lemma~\ref{lemmeCourbureRestePositive} once again, which proves that there exists $r > 0$ such that for all $\epsilon \in (0, \epsilon_0)$, $\gamma_N > 0$ on $\setof{(q, p) \in T^1\Sigma_\epsilon}{ q \in B(q_0, r), ~ \abs{p_x} \leq \alpha}$.
\end{proof}

\begin{prop} \label{propCourbure}

Choose $q_0 \in \Sigma$ such that $\pi(q_0) \in \partial D$, and assume that $N(q_0)$ is directed towards the exterior of the billiard table $D$, and (up to a rotation of axis $e_z$) that $N(q_0) = -e_y$. Write $q_0^\epsilon = f_\epsilon(q_0)$.

Then for all $\alpha \in (0, 1)$, there exists $r_0 > 0$ such that for all $r \leq r_0$ and for all $\nu \in (0,1)$:
\begin{equation} \label{eq2} \inf_{\substack{q \in V_\nu \cap B(q_0^\epsilon, r) \\ p \in T_q^1 \Sigma_\epsilon, ~ \abs{p_x} \leq \alpha}} \gamma_N^\epsilon \left(q, p \right) \underset{\epsilon \to 0}{\to} +\infty. \end{equation}

Moreover, under the additional assumption that the curvature of $\partial D$ is negative at $\pi(q_0)$, there exists $r_0 > 0$ such that for all $r \leq r_0$ and for all $\nu \in (0,1)$:
\begin{equation} \label{eq3} \limsup_{\epsilon \to 0} \sup_{\substack{q \in V_\nu \cap B(q_0^\epsilon, r)}} \gamma_-^\epsilon(q) < 0. \end{equation}
\end{prop}

To prove this proposition, we first prove a $2$-dimensional version in a particular case:

\begin{lemme} \label{lemmeCourbureCercle}
For all $\epsilon > 0$ consider the ellipse
\[ \mathcal E_\epsilon = \setof{(y, z) \in \mathbb R^2}{y^2 + \frac{z^2}{\epsilon^2} = 1}. \]
Define $N^\epsilon(q)$ as the unit normal vector of the ellipse at $q \in \mathcal E_\epsilon$, pointing towards the interior, and let
\[ W_\nu = \setof{z \in \mathcal E_\epsilon}{ \abs{N_z^\epsilon} < 1 - \nu }. \]
Then for all $\nu \in (0,1)$, if $K(q)$ denotes the curvature of $\mathcal E_\epsilon$ at $q$:
\[ \inf_{q \in W_\nu} K(q) \underset{\epsilon \to 0}{\to} +\infty. \]
\end{lemme}
\begin{proof}
We parametrize $\mathcal E_\epsilon$ by:
\[ t \mapsto \begin{pmatrix} \cos t  \\ \epsilon \sin t \end{pmatrix}. \]
Then the curvature is
\[ K(t) = \frac{\epsilon}{(\epsilon^2 \cos^2 t + \sin^2 t)^{3/2}} \]
while the unit normal vector is
\[ \begin{pmatrix}- \epsilon \cos t \\ - \sin t \end{pmatrix} \frac{1}{\sqrt{\epsilon^2 \cos^2 t + \sin^2 t}}. \]

If $\abs{N_z} < 1 - \nu$, then $\sin^2 t \geq (1 - \nu)^2 (\epsilon^2 \cos^2 t + \sin^2 t)$, whence $\tan^2 t \leq \frac{\epsilon^2 (1 - \nu)^2}{\nu(2-\nu)}$.

Therefore,
\[ K = \frac{\epsilon / \abs{cos^3 t}}{(\epsilon^2 + \tan^2 t)^{3/2}} \geq \frac{\epsilon}{\left(\epsilon^2 + \frac{\epsilon^2 (1 - \nu)^2}{\nu(2-\nu)}\right)^{3/2}} = \frac{(\nu(2-\nu))^{3/2}}{\epsilon^2} \]
which tends to $+\infty$ as $\epsilon \to 0$.
\end{proof}

\begin{proof}[Proof of Proposition~\ref{propCourbure}]
For each $q \in B(q, r)$, consider the curve $\Gamma_q$ resulting from the intersection of $\Sigma$ with the affine plane $(q, e_y, e_z)$ and the associated normal vector $n$.

With the notations of Definition~\ref{darboux}, let us show that we may choose $r$ small enough for $\abs{\prodscal{N}{T \wedge n}}$ to remain bounded away from $1$ for all small $\epsilon$ and all $q \in B(q, r)$. Since $N_x^1(q_0) = 0$, we may choose $r$ such that $N_x$ remains close to $0$ for $\epsilon = 1$. We know that $N_x$ decreases as $\epsilon$ decreases to $0$, so $N_x$ remains close to $0$ when $\epsilon \to 0$. Since $T \wedge n$ is colinear to $e_x$, this implies that $\abs{\prodscal{N}{T \wedge n}}$ remains close to $0$.

Now, let $\mathcal C$ be a circle tangent up to order $2$ to $\Sigma_1$ at $q$, parallel to $e_x^\bot$: the existence of such a circle is guaranteed, for a small enough $r$, by Assumption~\ref{courbureverticalenonnulle} of the theorem. This circle gives birth to a family $\mathcal E_\epsilon = f_\epsilon(\mathcal C)$ of ellipses which are tangent to $\Sigma_\epsilon$ at $f_\epsilon(q)$ up to order $2$. Lemma~\ref{lemmeCourbureCercle} tells us that as $\epsilon$ decreases to $0$, the curvature of $\mathcal E_\epsilon$ at $f_\epsilon(q)$ (which is the same as the curvature of $f_\epsilon(\Gamma_q)$ at $f_\epsilon(q)$) tends to infinity as long as $q \in V_\nu$, uniformly with respect to $q$. Together with Equation~\ref{eqk2}, this proves that
\[ \inf_{\substack{q \in V_\nu \cap B(q_0^\epsilon, r) \\ p \in T_q^1 \Sigma_\epsilon, ~ p_x = 0}} \gamma_N^\epsilon \left(q, p \right) \underset{\epsilon \to 0}{\to} +\infty. \]

To prove~(\ref{eq2}), let $\alpha \in (0, 1)$. Lemma~\ref{lemmeCourburePositive} applied to $q_0$ and $\frac{\alpha + 1}{2}$ gives us some $r_0$ and $\epsilon_0$ such that for all $q \in B(q_0, r_0)$ and all $p \in T_q^1 \Sigma_\epsilon$ such that $\abs{p_x} \leq \frac{\alpha+1}{2}$, $\gamma_N(q, p) > 0$. Since $\gamma_N(q, \cdot)$ is a quadratic form on the tangent space $T_q\Sigma_\epsilon$, which takes uniformly large values for $p \in T_q^1\Sigma_\epsilon \cap (e_x)^\bot$, we deduce that it also takes uniformly large values for $\abs{p_x} \leq \alpha$.

Finally, we prove~(\ref{eq3}): consider $q \in B(q_0^\epsilon, r) \cap V_\nu$, and $\Gamma$ a parametrization by arclength of $\setof{q' \in B(q_0, r) \cap V_\nu}{H(q') = H(q)}$. Since $H|_{Z_\delta}$ is a submersion (for any small enough $\delta$), the curvature of the curve $\pi \circ \Gamma$ is close to the curvature of $\partial D$ near $\pi(q)$, which is bounded away from zero. Moreover, the unit tangent vector of $\Gamma$ is bounded away from $e_z$ because of Assumption~\ref{courbureverticalenonnulle}, so the speed of $\pi \circ \Gamma$ is bounded away from zero, which implies that the curvature of $\Gamma$ itself is bounded away from $0$, uniformly with respect to $\epsilon$ and $q$. Moreover, $\prodscal{e_z}{T \wedge n}$ tends uniformly to $1$ as $\epsilon$ tends to $0$, so $\prodscal{N}{T \wedge n}$ is bounded away from $1$ in $V_\nu^\epsilon$. With Equation~\ref{eqk2}, this completes the proof of~(\ref{eq3}).
\end{proof}

As a direct consequence of Lemma~\ref{lemmeCourburePositive} and Proposition~\ref{propCourbure}, we obtain:
\begin{fact} \label{factCurvature}
If the walls of $D$ are negatively curved, then for any small enough $\delta$, the Gaussian curvature of $\Sigma_\epsilon$ in $Z_\delta$ is negative.
\end{fact}

In the following proposition, which is crucial for both Theorems~\ref{mainThm} and~\ref{thmAnosov}, we examine the nongrazing collisions with the walls of the billiards.

\begin{prop} \label{nonGrazing}
Consider a geodesic $(q^\epsilon(t), p^\epsilon(t))_{t \in \mathbb R}$ in $\Sigma_\epsilon$, for some $\epsilon > 0$. Denote by $t_c$ the time of the first bounce of the billiard trajectory starting from $\pi_*(q^\epsilon(0), p^\epsilon(0))$ (assume such a $t_c$ exists), let $(q^0(t), p^0(t))$ be the (unique) pullback of this trajectory by $\pi$ in $\Sigma_\epsilon$ for $t \in [0, t_c]$, and let $q_c = q^0(t_c)$. Assume that $N(q_c)$ is directed towards the exterior of the billiard table $D$, and (up to a rotation of axis $e_z$) that $N(q_c) = -e_y$. For any sufficiently small $\epsilon$, the trajectory outside $Z_\delta$ is close to the billiard trajectory, so the geodesic enters $Z_\delta$ at a time $t^\epsilon_\mathrm{in}$ close to $t_c$.

With these notations, for all $m > 0$ and $\alpha \in (0,1)$, there exist $l < 0$ and $\delta_0 > 0$, such that for all $\delta \leq \delta_0$, there exists $\epsilon_0 > 0$, such that for all $\epsilon \leq \epsilon_0$ and each geodesic $(q^\epsilon, p^\epsilon)$ in $\Sigma_\epsilon$ as above such that $\abs{p_x^\epsilon(0)} < \alpha$:
\begin{enumerate}
\item \label{ng1} $\sup_{t \in [0, t^\epsilon_\mathrm{in} + \sqrt{\delta}]} \abs{p^\epsilon_x(t) - p^\epsilon_x(0)} \leq m$;
\item \label{ng2a} the geodesic $(q^\epsilon(t), p^\epsilon(t))$ exits $Z_\delta$ at a time $t_\mathrm{out} \leq t_\mathrm{in} +  \sqrt \delta$;
\item \label{ng2b} if the curvature of $\partial D$ is negative everywhere, and if $q^\epsilon(0) \not\in Z_\delta$, then \[ \int_{t^\epsilon_\mathrm{in}}^{t^\epsilon_\mathrm{out}} K(q^\epsilon(t)) \mathrm dt \leq l. \]
\end{enumerate}
In particular, the choice of the constants $l$, $\delta_0$ and $\epsilon_0$ does not depend on the choice of the geodesic $(q^\epsilon, p^\epsilon)$.
\end{prop}

\begin{figure}
\centering
\begin{tikzpicture}
	\clip (-4,-2.5) rectangle (4,4);
    \path[fill=black!20] (-4,-4) rectangle (4,4);
    \path[draw=black, thick, fill=white] (0,-10) circle (10);
    \path[draw=black, dashed] (0,-10) circle (12);
    \path[draw=black, dashed] (0,-10) circle (10.5);
    \draw[draw=black, ->] (0,0) -- (0,-2) node[right] {$N(q_c)$};
    \node[circle, fill=black, inner sep=0pt, minimum size=3pt] (qc) at (0,0) {};
    \node (n1) at (-3.5,-0.4) {$V_\nu$};
    \node (n2) at (-3.5,1.2) {$Z_\delta$};
    \node (n3) at (-0.3,-0.3) {$q_c$};
    \path[draw=black, thick, densely dotted, fleche=0.8:black] (-4,4) -- (0,0) -- (4,4);
    \path[draw=black, fleche=0.2:black, fleche=0.8:black] (-4,4) .. controls (0,0) .. (0.4,{-10+sqrt(10^2-0.4^2)})  .. controls (0.8,{-10+sqrt(10^2-0.8^2)}) .. (4.8,4);
    \draw[draw=black, ->] (-3,-2) -- (-2,-2) node[below] {$e_x$};
    \draw[draw=black, ->] (-3,-2) -- (-3,-1) node[left] {$e_y$};
\end{tikzpicture}
\caption{The projection of the geodesic onto the billiard (solid line) is close to the billiard trajectory (dotted line).}
\end{figure}
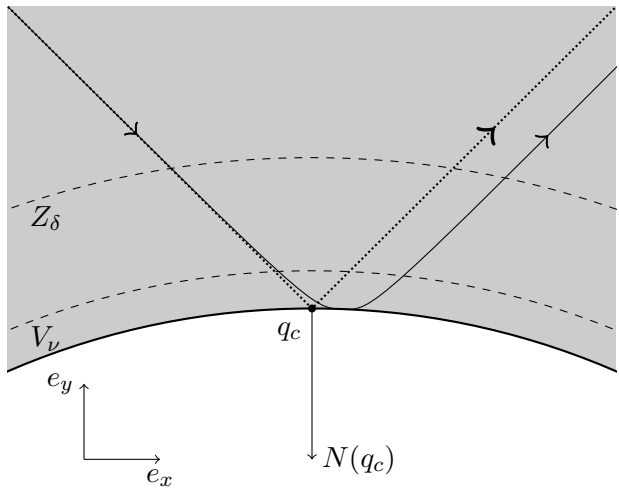

\begin{proof}
Let us prove the Statement (\ref{ng1}). We shall often write $q, p$ for $q^\epsilon(t), p^\epsilon(t)$ to simplify the notations.

Outside $Z_\delta$, the geodesic flow converges uniformly to the billiard flow, so we only need to consider $\sup_{t \in [t_\mathrm{in}, t_\mathrm{in} + \sqrt{\delta}]} \abs{p^\epsilon_x(t) - p^\epsilon_x(t_\mathrm{in})}$.

Let $t^1 = \inf \setof{t \in [t_\mathrm{in}, t_\mathrm{in}+\sqrt \delta]}{\prodscal{DN(q) \cdot p}{p} \leq 0}$ (or $t^1 = t_\mathrm{in} + \sqrt \delta$ if this set is empty), and consider $t \in [t_\mathrm{in}, t^1]$ (thus, $\prodscal{DN(q) \cdot p}{p} \geq 0$ at time $t$).

The geodesics $(q(t), p(t))$ follow the geodesic equation:
\[ \dot{p} = - N(q) \prodscal{DN(q) \cdot p}{p}, \]

which gives us the following estimates:
\[ \begin{aligned}
\abs{p_x(t) - p_x(t_\mathrm{in})} & \leq \int_{t_\mathrm{in}}^{t} \abs{\dot{p_x}}
\\ & \leq \int_{t_\mathrm{in}}^{t} \abs{\frac{N_x}{N_y}} \abs{\dot {p_y}}
\\ & \leq \left( \sup_{B(q(t_\mathrm{in}), \sqrt \delta)} \abs{\frac{N_x}{N_y}} \right) \int_{t_\mathrm{in}}^t \abs{\dot {p_y}}
\end{aligned}  \]

For all sufficiently small $\delta$, the quantity $N_y$ is negative in $B(q(t_\mathrm{in}), \sqrt{\delta})$, thus $\dot {p_y} = -N_y(q) \prodscal{DN(q) \cdot p}{p}$ is nonnegative and:

\[ \abs{p_x(t) - p_x(t_\mathrm{in})} \leq \left( \sup_{B(q(t_\mathrm{in}), \sqrt \delta)} \abs{\frac{N_x}{N_y}} \right) \abs{p_y(t) - p_y(t_\mathrm{in})}. \]

We know that $\frac{N_x}{N_y} = \frac{N_x^1}{N_y^1}$ does not depend on $\epsilon$. Moreover, $q(t_\mathrm{in})$ is close to $q_c$ and $\frac{N_x^1}{N_y^1}(q_c) = 0$, so the quantity $\sup_{B(q(t_\mathrm{in}), \sqrt \delta)} \abs{\frac{N_x}{N_y}}$ is close to $0$. On the other hand, $\abs{p_y(t) - p_y(t_\mathrm{in})}$ remains bounded since the geodesic has unit speed, which concludes the proof of Statement (\ref{ng1}) for $t \in [t_\mathrm{in}, t^1]$.

To extend the result to $t \in [t_\mathrm{in}, t_\mathrm{in} + \sqrt \delta]$, we prove that in fact $t^1 = t_\mathrm{in} + \sqrt \delta$: assume that $t^1 \neq t_\mathrm{in} + \sqrt \delta$. Then $p_x(t)$ remains close to $p_x(t_\mathrm{in})$ for $t \in [t_\mathrm{in}, t^1]$, so it remains bounded away from $1$, but $\gamma_N (q, p) = \prodscal{DN(q) \cdot p}{p} \leq 0$ at $t^1$. Thus, there is a contradiction with Lemma~\ref{lemmeCourburePositive}, and Statement (\ref{ng1}) is proved.

Now, let us prove Statement (\ref{ng2a}). We introduce the parameter $\nu$ and fix the parameters in the following way: first fix a small $\delta$, then a small $\nu$, and finally a small $\epsilon$.

Let us show that the boundaries of $Z_\delta$ and $V_\nu$ near $q_c$ are nearly parallel to the $e_x$ axis. From Fact~\ref{submersion}, the levels of $\left(H|_{Z_\delta}\right)^{-1}(a)$ are smooth curves. Moreover, for a sufficiently small $a \in [-1, 1]$, near $q_c$, the $y$-coordinates of the unit tangent vectors to $H^{-1}(a)$ remain small, while the $x$-coordinates are bounded away from zero. In particular, this applies to the boundary of $Z_\delta$, but also to the boundary of $V_\nu$, which is a level of $H$, since $N_z^\epsilon$ depends only on $H$ and $\epsilon$ (see Definition~\ref{defNormal}).

Outside $V_\nu$, $\abs{p_z}$ is bounded by $\sqrt{\nu(1-\nu)}$. Since $p_x$ is bounded away from $1$ and $p_x^2 + p_y^2 + p_z^2 = 1$, we deduce that $p_y$ remains bounded away from zero, uniformly with respect to $\delta$, $\nu$ and $\epsilon$, for all sufficiently small $\nu$. In particular, $p_y$ does not change sign in $Z_\delta \setminus V_\nu$, so it is only possible to enter $V_\nu$ once with $p_y < 0$ and exit once with $p_y > 0$. Thus, the geodesic can enter $V_\nu$ at most once.

There remains to show that the time spent in each zone is small.

For any $q \in \pi^{-1}(\Int D) \cap \Sigma_\epsilon$, it is natural to define $\frac{\partial H}{\partial x}$ as $DH_q(p)$, where $p$ is the (unique) speed vector in $T_q\Sigma$ such that $D_q\pi(p) = e_x$. We also define $\frac{\partial H}{\partial y}$ in the same way.

Outside $V_\nu$ (therefore outside $\pi^{-1}(\partial D)$) we write:
\[ \dot H = \frac{\partial H}{\partial x} p_x + \frac{\partial H}{\partial y} p_y. \]

Since the levels of the submersion $H$ are nearly parallel to $e_x$, $\frac{\partial H / \partial x}{\partial H / \partial y}$ is close to $0$ near $q_c$. Outside $V_\nu$, with the fact that $p_y$ is bounded away from $0$, this proves that $\dot H$ is bounded away from $0$, so the time spent in $Z_\delta \setminus V_\nu$ is $O(\delta)$.

In $V_\nu$ we have
\[ \dot{p_y} = - N_y \prodscal{ DN(q) \cdot p } {p}. \]

Fix $\delta, \nu > 0$. Since $N_z$ is bounded away from $1$, and $\frac{N_x}{N_y}$ is close to zero, we deduce that $N_y$ is bounded away from zero. Moreover, by Proposition~\ref{propCourbure}, $\prodscal{ DN(q) \cdot p } {p} \underset{\epsilon \to 0}{\to} +\infty$ uniformly in $V_\nu$, so $\dot{p_y} \underset{\epsilon \to 0}{\to} +\infty$. Since $p_y$ is bounded, this implies that the time spent in $V_\nu$ tends to $0$ as $\epsilon \to 0$. Thus, the total time spent in each zone is $O(\delta)$, so for any small enough $\delta$, $t_\mathrm{out} \leq t_\mathrm{in} +  \sqrt \delta$ (Statement~\ref{ng2a}).

If the curvature of $\partial D$ is negative everywhere, and $q^0(0) \not\in Z_\delta$, then the geodesic has the following behavior: it enters $Z_\delta$ with $p_y < 0$, then enters $V_\nu$ with $p_y < 0$. In $V_\nu$, $p_y$ changes sign, then the geodesic exits $V_\nu$ and finally, exits $Z_\delta$. Therefore, writing $t_2$ and $t_3$ the entry and exit times in $V_\nu$, since $K$ is negative in $Z_\delta$ (see Fact~\ref{factCurvature}):
\[ \begin{aligned} \int_{t_2}^{t_3} K = \int_{t_2}^{t_3} \gamma_+^\epsilon \gamma_-^\epsilon \leq \left(\sup_{V_\nu} \gamma_-^\epsilon\right) \int_{t_2}^{t_3} \gamma_+^\epsilon & \leq \left(\sup_{V_\nu} \gamma_-^\epsilon\right) \int_{t_2}^{t_3} - N_y \prodscal{ DN(q) \cdot p } {p} \\ & = \left(\sup_{V_\nu} \gamma_-^\epsilon\right) (p_y(t_3) - p_y(t_2)). \end{aligned} \]

With Proposition~\ref{propCourbure}, this proves that $\int_{t_2}^{t_3} K$ is bounded away from $0$: Statement (\ref{ng2b}) is proved.
\end{proof}

\paragraph{End of the proof of Theorem~\ref{mainThm}.} To prove the local uniform convergence, we introduce a family of elements $(t^\epsilon, q^\epsilon(0), p^\epsilon(0)) \in f_\epsilon(A)$ with parameter $\epsilon$, and assume that $\left(t^\epsilon, (f_\epsilon)_*^{-1}(q^\epsilon(0), p^\epsilon(0))\right)$ has a limit $(t^0, q^0(0), p^0(0)) \in A$ as $\epsilon \to 0$. The geodesic of $\Sigma_\epsilon$ starting at $(q^\epsilon(0), p^\epsilon(0))$ is written $(q^\epsilon(t), p^\epsilon(t))_{t \in \mathbb R}$. We want to show that $\pi_* (q^\epsilon(t^\epsilon), p^\epsilon(t^\epsilon))$ tends to $\Psi_{t^0} \circ \pi_* (q^0(0), p^0(0))$. Since the billiard trajectory $\Psi_t \circ \pi_* (q^0(0), p^0(0))$ experiences only a finite number of bounces in any finite time interval, we may assume that the trajectory for $t \in [0, t^0]$ has only one bounce\footnote{If the billiard trajectory has no bounce at all, then the geodesic remains outside of $Z_\delta$ and the convergence is clear.}, at a time $t_c$. As in Proposition~\ref{nonGrazing}, let $(q^0(t), p^0(t))$ be the (unique) pullback of this trajectory by $\pi$ in $\Sigma_\epsilon$ for $t \in [0, t_c]$, and let $q_c = q^0(t_c)$. Assume that $N(q_c)$ is directed towards the exterior of the billiard table $D$, and (up to a rotation of axis $e_z$) that $N(q_c) = -e_y$. The geodesic $(q^\epsilon(t), p^\epsilon(t))$ enters $Z_\delta$ at some time $t^\epsilon_\mathrm{in}$ and exits at some time $t^\epsilon_\mathrm{out}$, and the only difficulty to prove the convergence is located between these two times, since $g_\epsilon$ converges uniformly to a flat metric outside $Z_\delta$.

Since $p^0_x(0) < 1$, Proposition~\ref{nonGrazing} shows that
\[ \lim_{\delta \to 0} \lim_{\epsilon \to 0} \abs{p_x(t_\mathrm{out}) - p_x(t_\mathrm{in})} = 0. \]
Moreover, for all $\delta > 0$, $\lim_{\epsilon \to 0} p_z(t_\mathrm{in}) = \lim_{\epsilon \to 0} p_z(t_\mathrm{out}) = 0$, and since the geodesic has unit speed,
\[ \lim_{\delta \to 0} \lim_{\epsilon \to 0} \abs{p_y(t_\mathrm{in})} = \lim_{\delta \to 0} \lim_{\epsilon \to 0} \abs{p_y(t_\mathrm{out})}. \]

We have already seen that the geodesic enters $Z_\delta$ with $p_y < 0$ and exits with $p_y > 0$. Thus:
\[ \lim_{\delta \to 0} \lim_{\epsilon \to 0} p_y(t_\mathrm{in}) = - \lim_{\delta \to 0} \lim_{\epsilon \to 0} p_y(t_\mathrm{out}). \]

Proposition~\ref{nonGrazing} also states that $\lim_{\delta \to 0} \lim_{\epsilon \to 0} \abs{t_\mathrm{out} - t_\mathrm{in}} = 0$.

Thus, the limiting trajectory satisfies the billiard reflection law and the uniform convergence is proved.

\section{Proof of Theorem~\ref{thmAnosov}} \label{sectThmAnosov}

In this section, the walls of the billiard are assumed to be concave, and the billiard has \emph{finite horizon}. The following lemma gives an important consequence of the second property.

\begin{lemme} \label{lemmaFiniteHorizon}
Let $D$ be a billiard in $\mathbb T^2$ whose walls are negatively curved. Assume that $D$ has finite horizon ($D$ contains no geodesic of $\mathbb T^2$ with infinite lifetime in the past and the future). Then, there is an $\eta > 0$, a time $t^{\mathrm{max}}$ and an angle $\phi_0$ such that every curve of length $t^\mathrm{max}$ in $\mathbb T^2$, which is $\eta$-close to a straight line in the $C^1$ metric, hits at least once the boundary with an angle $\geq \phi_0$.
\end{lemme}
\begin{proof}
Assume that the conclusion of the lemma is false. Then there are curves $\Gamma_n : [-n, n] \to \mathbb T^2$ which do not hit the boundary with an angle greater than $\frac{1}{n}$, and which are $\frac{1}{n}$-close to geodesics in the $C^1$ metric. By a diagonal argument, one may extract a subsequence which converges to a geodesic $\Gamma : \mathbb R \to \mathbb T^2$ which does not hit the boundary with an angle greater than $0$, so that $\Gamma$ remains in $D$.
\end{proof}

As another consequence of the concavity of the walls, we may assume that the principal curvatures satisfy $\abs{\gamma_-^\epsilon} \leq \abs{\gamma_+^\epsilon}$ in $Z_\delta$ (with Proposition~\ref{propCourbure}). We write $\kappa(\delta, \epsilon) = max_{q \not\in Z_\delta} \left(\abs{\gamma_+^\epsilon(q)}, \abs{\gamma_-^\epsilon(q)}\right)$. Notice that for all $\delta > 0$, $\kappa(\delta, \epsilon) \underset{\epsilon \to 0}{\to} 0$. Later we will simply write $\kappa$ for $\kappa(\delta, \epsilon)$. We also define
\[ W_\kappa = \setof{q \in \Sigma_\epsilon}{ K(q) \leq - \kappa }. \]
Notice that for any fixed $\delta > 0$, there exists $\epsilon_0 > 0$ such that for $\epsilon \leq \epsilon_0$, $W_\kappa \subseteq Z_\delta$.

In the following proposition, we determine what remains of Proposition~\ref{nonGrazing} when the geodesics are not assumed to be nongrazing, but when instead they are assumed to undergo little curvature.

\begin{prop} \label{grazing}
Consider a geodesic $(q^\epsilon(t), p^\epsilon(t))_{t \in \mathbb R}$ in $\Sigma_\epsilon$, for some $\epsilon > 0$. Define $t^\epsilon_\mathrm{in}$ as the first time at which the geodesic $(q^\epsilon(t), p^\epsilon(t))$ enters $Z_\delta$ and assume that $t^\epsilon_\mathrm{in} + \sqrt[3] \delta < 1$. As before, choose the orientation of the normal vector $N$ such that it points towards the exterior of the billiard table at its boundary near $q_\mathrm{in}$. Up to a rotation of axis $e_z$, we may require that $N_x(q_\mathrm{in}) = 0$ and $N_y(q_\mathrm{in}) < 0$.

With these notations, for all $m > 0$, there exists $\delta_0 > 0$, such that for all $\delta \leq \delta_0$, there exists $\epsilon_0 > 0$, such that for all $\epsilon \leq \epsilon_0$ and each geodesic $(q^\epsilon, p^\epsilon)$ as above such that $\int_0^1 \abs{K(q^\epsilon(t))} \leq 3\kappa^2$:
\begin{enumerate}
\item \label{stat1} $\inf_{t \in [0, t_\mathrm{in} + \sqrt[3]\delta]} (p_y(t) - p_y(0)) \geq - 2 \sqrt{\kappa}$;
\item \label{stat2a} $\sup_{t \in [0, t_\mathrm{in} + \sqrt[3]\delta]} \abs{p_x(t) - p_x(0)} \leq m$;
\item \label{stat2b} denoting by $Z_\delta^0$ the connected component of $Z_\delta$ containing $q(t_\mathrm{in})$, there exists $t_\mathrm{out} \in [t_\mathrm{in}, t_\mathrm{in} + \sqrt[3]\delta]$, at which the geodesic exits $Z_\delta^0$, and does not come back to $Z_\delta^0$ before visiting another component of $Z_\delta$.
\end{enumerate}
In particular, the choice of $\delta_0$ and $\epsilon_0$ does not depend on the choice of the geodesic $(q^\epsilon(t), p^\epsilon(t))$.

\end{prop}

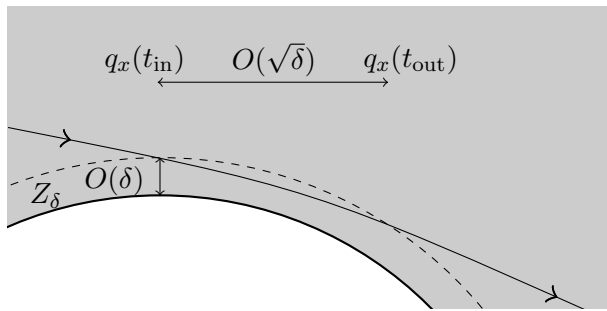
\begin{figure}[ht!]
\centering
\begin{tikzpicture}
	\clip (-4,-2) rectangle (4,2);
    \path[fill=black!20] (-4,-4) rectangle (4,4);
    \path[draw=black, thick, fill=white] (-2,-5.5) circle (5);
    \path[draw=black, dashed] (-2,-5.5) circle (5.5);
    \draw[draw=black, <->] (-2.03,1) -- (1,1) node[right] {};
    \draw[draw=black, <->] (-2,-0.5) -- (-2,0) node[right] {};
    \node at (-2.2,1.3) {$q_x(t_\mathrm{in})$};
    \node at (1.3,1.3) {$q_x(t_\mathrm{out})$};
    \node at (-2.6,-0.3) {$O(\delta)$};
    \node at (-0.5,1.3) {$O(\sqrt\delta)$};
    \node at (-3.5,-0.5) {$Z_\delta$};
    \path[draw=black, fleche=0.1:black, fleche=0.9:black] (-4,0.4) .. controls (0,-0.4) .. (4,-2.2);
\end{tikzpicture}
\caption{The geodesic exits $Z_\delta$ before a time which is $O(\sqrt{\delta})$, because $p_y(t) - p_y(0)$ is bounded from below by a small negative number.} \label{figGrazing}
\end{figure}

\begin{proof}
As before, we only need to consider what happens for $t \geq t_\mathrm{in}$, as the metric tends to a flat metric outside $Z_\delta$.

To prove the first statement, writing $\dot {p_y} = \dot {p_y}^+ - \dot {p_y}^-$, with $\dot {p_y}^+ = \max (\dot {p_y}, 0)$ (positive part) and $\dot {p_y}^-(t) = - \min (\dot {p_y}, 0)$ (negative part), it suffices to show that $\int_{t_\mathrm{in}}^t \dot {p_y}^-$ is (uniformly) close to $0$. We divide this integral into two parts. In the part where $q \in W_{\kappa}$, the quantity $\dot{p_y} = -N_y(q) \prodscal{DN(q) \cdot p}{p}$ is bigger than $-\sqrt{K(t)}$ (because $\abs{\gamma_-^\epsilon} \leq \abs{\gamma_+^\epsilon}$), so it is bigger than $-K(t) - 1$. The time spent in $W_{\kappa}$ is smaller than $3\kappa^2/\kappa$, and the integral of $\abs{K(t)}$ is smaller than $3 \kappa^2$. In the part where $q \not\in W_{\kappa}$, $\prodscal{DN(q) \cdot p}{p}$ is bigger than $-\sqrt{\kappa}$. Thus, $\int_{t_\mathrm{in}}^t \dot {p_y}^- \leq 3 \kappa^2 + 3 \kappa + \sqrt{\kappa} (t_\mathrm{in} + \sqrt[3]\delta) \leq 2 \sqrt{\kappa}$, and Statement~\ref{stat1} is proved.

For all $t \in [0, \sqrt[3]\delta]$, we may write, as in the proof of Proposition~\ref{nonGrazing}:
\[ \abs{p_x(t) - p_x(t_\mathrm{in})} \leq \left( \sup_{B(q(t_\mathrm{in}), \sqrt[3] \delta)} \abs{\frac{N_x}{N_y}} \right) \int_{t_\mathrm{in}}^t \abs{\dot {p_y}}. \]
Since $\frac{N_x}{N_y}$ does not depend on $\epsilon$, $\left( \sup_{B(q(t_\mathrm{in}), \sqrt[3] \delta)} \abs{\frac{N_x}{N_y}} \right)$ is close to $0$. Moreover,
\[ \int_{t_\mathrm{in}}^t \abs{\dot {p_y}} = \int_{t_\mathrm{in}}^t \left(\dot {p_y} + 2 \dot {p_y}^-\right) \leq \abs{p_y(t) - p_y({t_\mathrm{in}})} + 2 \int_{t_\mathrm{in}}^t \dot {p_y}^-. \]
The term $\abs{p_y(t) - p_y({t_\mathrm{in}})}$ is bounded by $2$, and $\int_{t_\mathrm{in}}^t \dot {p_y}^-$ is close to $0$, so $\int_{t_\mathrm{in}}^t \abs{\dot {p_y}}$ is bounded, which proves Statement~\ref{stat2a}.

Finally, to prove Statement~\ref{stat2b}, fix any $\alpha \in (0,1)$. Lemma~\ref{nonGrazing} implies that all trajectories such that $\abs{p_x(0)} < \alpha$ exit $Z_\delta$ before $t = t_\mathrm{in} + \sqrt{\delta}$. For the other trajectories, Statement~\ref{stat2a} implies that $p_x$ remains bounded away from $0$. Together with Statement~\ref{stat1} and the uniform concavity of the walls of the billiard table, this implies that the geodesic must exit $Z_\delta^0$ definitively before a time which is $O(\sqrt \delta)$ (see Figure~\ref{figGrazing}).
\end{proof}

\begin{lemme} \label{straightGeodesic}
For all $m > 0$, there exists $\epsilon_0 > 0$, such that for all $\epsilon \leq \epsilon_0$, and all geodesics $(q^\epsilon(t), p^\epsilon(t))_{t \in \mathbb R}$ such that $\int_0^1 \abs{K(q^\epsilon(t))} \leq 3\kappa^2$,
\[ \sup_{t \in \left[\frac{1}{3}, \frac{2}{3}\right]} \norm{p(t) - p(1/3)} \leq m. \]
In particular, the choice of $\epsilon_0$ does not depend on the choice of the geodesic.
\end{lemme}

\begin{proof}
Outside $Z_\delta$, $\dot{p_x}$ vanishes as $\epsilon \to 0$. Each time that the geodesic enters or exits $Z_\delta$, Proposition~\ref{nonGrazing} implies (with the choice of $\alpha \in (0,1)$ close to $1$) that $p$ is nearly tangent to the boundary of the billiard table (otherwise, the geodesic undergoes strong negative curvature after the entry or before the exit, which is why we consider only the interval $\left[\frac{1}{3}, \frac{2}{3}\right]$). Moreover, Proposition~\ref{grazing} implies that the time spent in $Z_\delta$ is small. Thus, the exit point is near the entrance point and, from Statement~\ref{ng2a} of Proposition~\ref{grazing}, the speed vector $p$ is almost preserved. Then, the geodesic goes to visit another component of $Z_\delta$, so there is an upper bound on the number of times that it enters $Z_\delta$. Thus, the total change in $p$ is uniformly small.
\end{proof}

\begin{proof}[End of the proof of Theorem~\ref{thmAnosov}]
To show that the flow has the Anosov property, we consider a small $\delta$, a small $\epsilon$, and a geodesic $(q^\epsilon(t), p^\epsilon(t))_{t \in \mathbb R}$ in $\Sigma_\epsilon$, and examine the Ricatti equation:
\[ 
\left\{
\begin{aligned}
u (0) & = 0
\\ u'(t) & = - K^\epsilon(q^\epsilon(t)) - u^2(t).
\end{aligned} \right. \]
It suffices to show that $u(1)$ is positive and bounded away from $0$, uniformly with respect to the choice of the geodesic (see for example~\cite{donnay2003anosov} or~\cite{magalhaes2013geometry}). In the following, we write $K(t) := K^\epsilon(q^\epsilon(t))$.

Applying a homothety to $\Sigma$ if necessary, we may assume that $t^{\mathrm{max}}$ given by Lemma~\ref{lemmaFiniteHorizon} is less than $\frac{1}{3}$. If $\int_0^1 \abs{K(t)} \mathrm dt \leq 3\kappa^2$, then Lemma~\ref{straightGeodesic} tells us that, for any small enough $\epsilon$, $(\pi_*(q(t), p(t)))_{t \in \left[\frac{1}{3}, \frac{2}{3} \right]}$ is $C^1$-close to a straight line in $\mathbb T^2$, which contradicts Lemma~\ref{lemmaFiniteHorizon}. Thus, there exists $\epsilon_0 > 0$ such that for all $\epsilon \leq \epsilon_0$, $\int_0^1 \abs{K(t)} \mathrm dt \geq 3\kappa^2$. Since $K \leq 0$ in $Z_\delta$ and $\abs{K} \leq \kappa^2$ outside $Z_\delta$, we deduce that $K \leq \kappa^2$ in $\Sigma_\epsilon$. Therefore, considering the positive and negative parts of $K$, \[ \int_0^1 K = \int_0^1 (K^+ - K^-) = - \int_0^1 \abs{K} + 2 \int_0^1 K^+ \leq -3 \kappa^2 + 2 \kappa^2 \leq - \kappa^2. \]

Now, let us show that $u(1) \geq \kappa^2/2$, which will end the proof. To do this, we assume that $u(1) < \kappa^2/2$ and show that for all $t \in [0, 1]$, $\abs{u(t)} \leq 2 \kappa^2$. Let $t^1 = \sup \setof{t \in [0, 1]}{u(t) \geq 2 \kappa^2}$ (or $t^1 = 0$ if this set is empty). For $t \in [t^1, 1]$, $u'(t) = - K (t) - u^2(t) \geq - K(t) - 4 \kappa^4$, so $u(1) - u(t^1) \geq - \int_{t^1}^1 K(t) \mathrm dt - 4 \kappa^4$, whence \[u(t^1) \leq u(1) + \int_{t^1}^1 K(t) \mathrm dt + 4 \kappa^4 \leq \kappa^2/2 + \kappa^2 + 4 \kappa^4 < 2 \kappa^2. \] This implies, with the definition of $t^1$, that $t^1 = 0$ and $u(t) \leq 2 \kappa^2$ for all $t \in [0, 1]$. Thus,
\[ u(1) \geq u(0) - \int_0^1 K(t) - 4 \kappa^4 \geq \kappa^2 - 4 \kappa^4 \geq \kappa^2/2, \]
a contradiction.
\end{proof}

\section{Application to linkages} \label{sectLinkages}

The aim of this section is to prove that the configuration space of the linkage described in Theorem~\ref{thmLinkage} is isometric to an immersed surface in $\mathbb T^2 \times \mathbb R$ which satisfies the $4$ assumptions of Theorem~\ref{thmAnosov}.

The configuration space $\mathrm{Conf}(\mathcal L)$ is the set of all $(a, b, c, d, e, f, g) \in \mathbb R^7$ such that:
\[ (a+2)^2 + f^2 = (b-2)^2 + g^2 = 1; \]
\[ (a-d)^2 + e^2 = (b-d)^2 + e^2 = l^2; \]
\[ d^2 + (c-e)^2 = r^2. \]

Notice that $(a+2, f)$ and $(b-2, g)$ lie in the unit circle $\mathbb T \subseteq \mathbb R^2$. Thus, $\mathrm{Conf}(\mathcal L)$ is in fact a subset of $\mathbb T^2 \times \mathbb R^3$ and any of its elements may be written $(\theta, \phi, c, d, e)$, with the identification $a = -\cos \theta - 2$, $f = \sin \theta$, $b = \cos \phi + 2$, $g = \sin \phi$.

\begin{fact} \label{graphe1}
For all $\mathcal C_0 \in \mathrm{Conf}(\mathcal L)$ such that $e \neq 0$ and $e \neq c$, $\mathrm{Conf}(\mathcal L)$ is locally a smooth graph above $\theta$ and $\phi$ near $\mathcal C_0$. More precisely, there exists a neighborhood $U$ of $\mathcal C_0$ in $\mathrm{Conf}(\mathcal L)$, an open set $V$ of $\mathbb T^2$ and a smooth function $F : V \to \mathbb R^3$ such that
\[ U = \setof{(\mathcal D, F(\mathcal D))}{\mathcal D \in V}. \]
\end{fact}
\begin{proof}
The function $F$ is given by the following formulae:
\[ \small \begin{aligned}d & = \frac{-\cos \theta+\cos \phi}{2};
\\ e & = \pm \sqrt{l^2 - \left(\frac{\cos \theta+\cos \phi}{2}+2\right)^2} = \pm \sqrt{l^2 - \left(\frac{\cos \theta+\cos \phi+4}{2}\right)^2};
\\ c & = e \pm \sqrt{r^2 - \left(\frac{\cos \theta-\cos \phi}{2}\right)^2} = \pm \sqrt{l^2 - \left(\frac{\cos \theta+\cos \phi + 4}{2}\right)^2} \pm \sqrt{r^2 - \left(\frac{\cos \theta-\cos \phi}{2}\right)^2} \end{aligned} \]
where the choices of the signs are made according to $\mathcal C_0$.
\end{proof}

\begin{fact} \label{graphe2}
\begin{enumerate}[a)]
\item \label{graphe2a} For all $\mathcal C_0 \in \mathrm{Conf}(\mathcal L)$ such that $(-\cos \theta-2, 0)$, $(d, e)$ and $(0, c)$ are not aligned, and such that $\phi \neq 0 \mod \pi$, $\mathrm{Conf}(\mathcal L)$ is locally a smooth graph above $\theta$ and $c$ near $\mathcal C_0$.
\item \label{graphe2b} For all $\mathcal C_0 \in \mathrm{Conf}(\mathcal L)$ such that $(\cos \phi+2, 0)$, $(d, e)$ and $(0, c)$ are not aligned, and such that $\theta \neq 0 \mod \pi$, $\mathrm{Conf}(\mathcal L)$ is locally a smooth graph above $\phi$ and $c$ near $\mathcal C_0$.
\end{enumerate}
\end{fact}
\begin{proof}
By symmetry, we only need to prove the first statement. The idea of the proof is the same as for Fact~\ref{graphe1}: on the one hand, the numbers $d$ and $e$ are obtained as the simple roots of a polynomial of degree $2$, so they vary smoothly with respect to $\theta$ and $c$; on the other hand, $\phi = \pm \arccos(2d + \cos \theta)$ where the choice of the sign is made according to $\mathcal C_0$.
\end{proof}

\begin{fact} \label{graphe3} For all $\mathcal C_0 \in \mathrm{Conf}(\mathcal L)$, $\mathrm{Conf}(\mathcal L)$ is locally a smooth graph near $\mathcal C_0$:
\begin{enumerate}
\item either above $\theta$ and $\phi$,
\item or above $\theta$ and $c$,
\item or above $\phi$ and $c$.
\end{enumerate} \end{fact}
\begin{proof}
Assume the opposite. Then the hypotheses of Fact~\ref{graphe2} are not satisfied. If $\phi = \theta = 0 \mod \pi$, then $\phi = \theta \mod 2 \pi$ because $r < 1/2$; then $\phi = \theta = \pi \mod 2\pi$ because $l < 3$, but this implies that $e \neq 0$ and $e \neq c$, so Fact~\ref{graphe1} applies, which is impossible. Therefore, by symmetry, we may assume that $(-\cos \theta-2, 0)$, $(d, e)$ and $(0, c)$ are aligned. Now, with Fact~\ref{graphe1}, we have either $e = 0$ or $e = c$. In both cases, $(-\cos \theta-2, 0)$, $(d, e)$ and $(0, c)$ are all on the line $y = 0$, which contradicts the fact that $l + r > 3$.
\end{proof}

Fact~\ref{graphe3} implies in particular that $\mathrm{Conf}(\mathcal L)$ is a smooth submanifold of $\mathbb T^2 \times \mathbb R^3$.

As explained in the introduction, $\mathrm{Conf}(\mathcal L)$ is endowed with the metric which corresponds to its kinetic energy (recall that the masses of the vertices are $\epsilon^2$ at $(0, c)$, $0$ at $(d, e)$, and $1$ everywhere else):
\[ g_\epsilon = da^2 + df^2 + db^2 + dg^2 + \epsilon^2 dc^2 = d\theta^2 + d\phi^2 + \epsilon^2 dc^2. \]

Fact~\ref{graphe3} shows that the metric $g_\epsilon$ is nondegenerate (although it is induced by a degenerate metric of $\mathbb T^2 \times \mathbb R^3$!), so with Fact~\ref{geodesicbehavior} the physical behavior of the linkage is the geodesic flow on $(\mathrm{Conf}(\mathcal L), g_\epsilon)$. Our aim is to show that it is an Anosov flow by applying Theorem~\ref{thmAnosov}.

Consider the projection onto the first coordinates:
\[ \begin{aligned} p: \mathbb T^2 \times \mathbb R^3 & \to \mathbb T^2 \times \mathbb R
\\ (\theta, \phi, c, d, e) & \mapsto (\theta, \phi, c). \end{aligned} \]

Again with Fact~\ref{graphe3}, $p|_{\mathrm{Conf}(\mathcal L)}$ is an immersion: ${\mathrm{Conf}(\mathcal L)}$ is isometric to a smooth surface $\Sigma$ immersed in $\mathbb T^2 \times \mathbb R$, endowed with the metric $g_\epsilon = d\theta^2 + d\phi^2 + \epsilon^2 dc^2$. We shall now call $z$ the third coordinate instead of $c$, to be consistent with the notations of Theorem~\ref{thmAnosov}.

Denote by $\pi: \mathbb T^2 \times \mathbb R \to \mathbb T^2$ the projection onto the first coordinates. The surface $\Sigma$ projects to a smooth billiard table:
\[ D = \pi(\mathrm{Conf}(\mathcal L)) = \setof{ (\theta, \phi) \in \mathbb T^2 }{ \abs{\cos \theta - \cos \phi} \leq 2r, \quad \cos \theta + \cos \phi \leq 2l - 4} \]
Its boundary has three connected components in $\mathbb T^2$: $\{\cos \theta - \cos \phi = 2r\}$, $\{-\cos \theta + \cos \phi = 2r\}$, and $\{\cos \theta + \cos \phi = 2l - 4\}$.

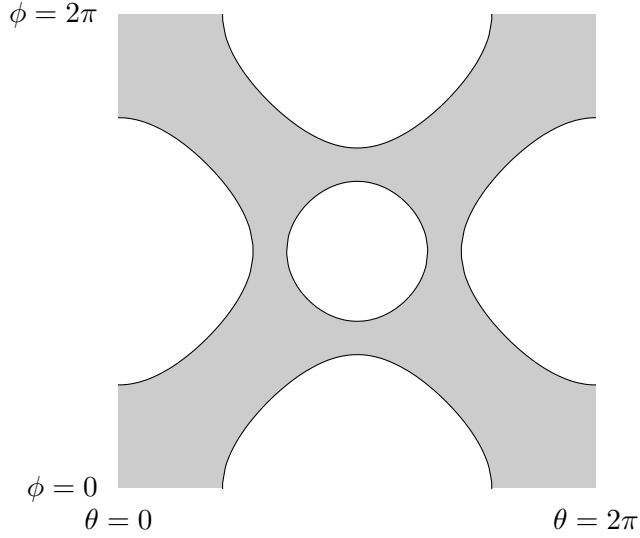
\begin{figure}[h!]
\begin{center}
\begin{tikzpicture}
  \draw (-pi, -pi) node[label=below:{$\theta = 0$}] {};
  \draw (pi, -pi) node[label=below:{$\theta = 2\pi$}] {};
  \draw (-pi, -pi) node[label=left:{$\phi = 0$}] {};
  \draw (-pi, pi) node[label=left:{$\phi = 2\pi$}] {};
  \clip (-pi, -pi) rectangle (pi, pi);
  \fill [color=black!20] (-pi, -pi) rectangle (pi, pi);

\obstacle{0.8}{-pi}{0}
\obstacle{0.8}{pi}{0}
\obstacle{0.8}{0}{pi}
\obstacle{0.8}{0}{-pi}
\obstacle{1.6}{0}{0}
\end{tikzpicture}
\caption{The billiard table $D$ (in grey) for $r = 0.4$ and $l = 2.8$. The billiard has negatively curved walls, which means that the obstacles are strictly convex.} \label{figBilliard}
\end{center}
\end{figure}

There remains to show that the immersed surface $\Sigma$ satisfies the 4 assumptions of Theorem~\ref{thmAnosov}. Assumption~\ref{pasverticalsaufaubord} is satisfied as a direct consequence of Fact~\ref{graphe1}. The following proposition proves Assumption~\ref{courbureverticalenonnulle}.
\begin{prop}
For all $q \in \pi^{-1}(\partial D) \cap \Sigma$, the curvature of $\Sigma \cap V$ is nonzero at $q$, where $V$ is a neighborhood of $q$ in the affine plane $q + \mathrm{Vect}(e_z, (T_q\Sigma)^\bot)$.
\end{prop}
\begin{proof}
Here we will assume that $\pi(q) \in \setof{(\theta, \phi) \in \mathbb T^2}{\cos \phi + \cos \theta = 2l - 4}$, but the proof is identical for the other components of $\partial D$. Let $F(\theta, \phi) = \left(\frac{\cos \theta + \cos \phi + 4}{2}\right)^2$. For any small $t \geq 0$, let $\theta(t) = q_\theta + N_\theta(q)t$, $\phi(t) = q_\phi + N_\phi(q)t$, and choose $z(t)$ of the form:
\[ z(t) = \pm \sqrt{l^2 - \left(\frac{\cos \theta(t)+\cos \phi(t) + 4}{2}\right)^2} \pm \sqrt{r^2 - \left(\frac{\cos \theta(t)-\cos \phi(t)}{2}\right)^2}\]
with a choice of the $\pm$ signs so that $(\theta(0), \phi(0), z(0)) = q$. Then for all small $t \geq 0$, $(\theta(t), \phi(t), z(t)) \in \Sigma$.

As $t$ tends to $0$, we may estimate:
\[ z(t) = \pm \sqrt{\left(\left.\frac{\mathrm{d}}{\mathrm{d}t}\right|_{t = 0} F(\theta(t), \phi(t))\right)t + O(t^2)} \pm \sqrt{r^2 - \left(\frac{\cos \theta(0)-\cos \phi(0)}{2}\right)^2} + O(t) \]
\[ (z(t) - z(0))^2 = \pm \left(\left.\frac{\mathrm{d}}{\mathrm{d}t}\right|_{t = 0}
 F(\theta(t), \phi(t))\right)t + o(t). \]
Notice that $\nabla F(\theta(0), \phi(0)) = - \begin{pmatrix}\sin \theta(0) \\ \sin \phi(0)\end{pmatrix} \left(\frac{\cos \theta + \cos \phi + 4}{2}\right)$ is nonzero (because $2 < l < 3$). Moreover, $\begin{pmatrix}\theta'(0) \\ \phi'(0) \end{pmatrix}$ is $\begin{pmatrix}N_\theta(q) \\ N_\phi(q)\end{pmatrix}$, which is colinear to $\nabla F(\theta(0), \phi(0))$, so
\[ \left.\frac{\mathrm{d}}{\mathrm{d}t}\right|_{t = 0} F(\theta(t), \phi(t)) = \begin{pmatrix}\theta'(0)\\ \phi'(0)\end{pmatrix} \cdot \nabla F(\theta(0), \phi(0)) \neq 0. \] This gives us:
\[ t \sim \pm \frac{1}{\left.\frac{\mathrm{d}}{\mathrm{d}t}\right|_{t = 0} F(\theta(t), \phi(t))} (z(t) - z(0))^2. \]

Hence, $t \mapsto z(t)$ has an inverse function $z \mapsto t(z)$ which has a nonzero second derivative at $t = 0$. Since $(z, t)$ are the coordinates in an affine (orthonormal) basis of $q + \mathrm{Vect}(e_z, (T_q\Sigma)^\bot)$, this implies that $\Sigma \cap V$ has nonzero curvature at $q$.
\end{proof}

The following proposition proves Assumption~\ref{billardcourburenegative}.

\begin{prop}
The walls of the billiard $D$ have negative curvature.
\end{prop}
\begin{proof}
In general, the curvature of the boundary of a set defined by the inequality $F(q) \leq C$, where $C \in \mathbb R$ is a constant, with the normal vector pointing inwards, is the divergence of the normalized gradient of $F$, namely:
\[ \nabla \cdot \frac{\nabla F}{\norm{\nabla F}}. \]

First consider the boundary of the set $\{ \cos \phi + \cos \theta \leq 2l-4 \}$. Here $F(\phi, \theta) = \cos \phi + \cos \theta$. Thus:

\[ \frac{\nabla F}{\norm{\nabla F}} = \frac{-1}{\sqrt{\sin^2 \theta + \sin^2 \phi}} \begin{pmatrix}\sin \theta \\ \sin \phi\end{pmatrix}. \]

Hence, the divergence of the normalized gradient has the same sign as:
\[ -\sin^2 \phi \cos \theta - \sin^2 \theta \cos \phi \]
which can be rewritten:
\[ -(2l-4)\cos^2\theta +(2l-4)^2\cos \theta - (2l-4). \]
This is a second order polynomial in $\cos \theta$ with discriminant $(2l-4)^2((2l-4)^2-4) < 0$ (here we use the assumption $l < 3$). Since $(2l-4) > 0$ (because $l > 2$), the polynomial is everywhere negative.

Now, consider the boundary of the set $\{ \cos \phi - \cos \theta \leq 2r \}$. This time, the divergence of the normalized gradient has the same sign as
\[ \sin^2 \phi \cos \theta - \sin^2 \theta \cos \phi \]
which can be rewritten
\[ -2r \cos^2\theta - 4r^2 \cos \theta - 2r. \]
This time, the discriminant is $16r^2(r^2-1)$, which is negative since $r < 1$.

The third wall is the boundary of the set $\{ \cos \theta - \cos \phi \leq 2r \}$. The divergence of the normalized gradient has the same sign as
\[ -\sin^2 \phi \cos \theta + \sin^2 \theta \cos \phi \]
which can be rewritten
\[ -2r \cos^2\theta + 4r^2 \cos \theta - 2r. \]
Again, the discriminant is $16r^2(r^2-1)$, which is negative.
\end{proof}

Finally, we prove Assumption~\ref{horizonfini}.
\begin{prop}
If $(l-2)^2 + r^2 < 1$ and $r < 1/2$, then $D$ has finite horizon.
\end{prop}
\begin{proof}
Assume that there exists a geodesic $(\theta(t), \phi(t))$ with infinite lifetime in the past and in the future.

First, we prove that the slope of the geodesic is $\pm 1$. We may assume that the slope is in $[-1, 1]$ (if not, exchange the roles $\theta$ and $\phi$): thus there is a time $t_0$ for which $\theta(t_0) = 0 \mod 2\pi$. Then the set $G = \setof{\phi(t) - \phi(t_0) \mod 2\pi}{t \in \mathbb R, ~ \theta(t) = 0 \mod 2\pi}$ is a subgroup of $\mathbb R / 2\pi \mathbb Z$. Moreover, for all $t \in G$, we have $\abs{\cos \theta(t) - \cos \phi(t)} \leq 2r$ so $\cos \phi(t) \geq 1 - 2r > 0$, so $G \subseteq \left(-\frac{\pi}{2}, \frac{\pi}{2}\right) \mod 2 \pi$, which means that $G$ is reduced to a single point: the slope is either $0$ or $\pm 1$ (since we assumed it is in $[-1, 1]$). If the slope is $0$, then $\abs{\cos \theta(t) - \cos \phi(t)} \leq 2r$ applied to a $t$ such that $\cos \theta(t) = -1$ gives us $\cos \phi(t) \leq -1 + 2r$, which is not compatible with $\cos \phi(t) \geq 1 - 2r > 0$ since $r < 1/2$, so the slope is in fact $\pm 1$.

Changing $\theta$ into $-\theta$ if necessary, we may assume that the slope is $1$. Thus, there exist $t_1$ and $t_2$ such that $\phi(t_1) + \theta(t_1) = \pi \mod 2\pi$ and $\phi(t_2) + \theta(t_2) = 0 \mod 2\pi$. We have $\theta(t_2) - \theta(t_1) = \phi(t_2) - \phi(t_1) \mod 2\pi$ (because the slope is $1$), so $\phi(t_2) - \phi(t_1) = \frac{\pi}{2} \mod \pi$, so $\cos \phi (t_2) \cos \phi(t_1) = - \sin \phi(t_2) \sin \phi(t_1)$. By taking the squares of both sides of this equality we obtain: \begin{equation} \label{eqthetaphi} \cos^2 \phi(t_1) + \cos^2 \phi(t_2) = 1. \end{equation}
We have $- \cos \theta(t_1) + \cos \phi(t_1) \leq 2r$ and $\cos \phi(t_2) + \cos \theta(t_2) \leq 2l - 4$, which implies that $- \cos \theta(t_1) = \cos \phi(t_1) \leq r$ and $\cos \theta(t_2) = \cos \phi(t_2) \leq l-2$. Injecting this in (\ref{eqthetaphi}), we obtain:
\[ r^2 + (l-2)^2 \geq 1, \]
which contradicts $r^2 + (l-2)^2 < 1$.
\end{proof}

\section{Acknowledgements}
I would like to thank Jos Leys for realizing Figures~\ref{figPhysical}, \ref{figConfSpace} and~\ref{figBilliardGeodesics}, as well as a video which is available on my website.

\bibliographystyle{alpha}
\bibliography{anosov-linkages}

\end{document}